\documentclass[reqno]{amsart} 
\usepackage{amsmath,amsfonts,amssymb, amsthm, amscd}
\usepackage{mathtools}
\usepackage{hyperref}
\usepackage{color}
\usepackage{mathabx} 
\usepackage{enumitem}

\newtheorem{theorem}{Theorem}
\newtheorem{lemma}[theorem]{Lemma} 
\newtheorem{remark}[theorem]{Remark}

\newtheorem{corollary}[theorem]{Corollary}

\usepackage[all,knot]{xy}
\xyoption{arc}
\usepackage{graphics}
\usepackage{enumitem}
\usepackage[all,knot]{xy}
\xyoption{arc}

\newcommand{ \Pcal }[1]{\mathbb{P}^{{#1}}}	
\newcommand{ \Mcal }[1]{\mathbb{M}^{{#1}}}	

\newcommand{ \Scal }{\mathbb{S}}			
\newcommand{ \elem }[2]{{e_{ ({#1},{#2})}}} 	

\newcommand{ \Qcal }{\mathcal{Q}}			

\newcommand{ \Acal }{\mathcal{A}}			
\newcommand{ \Ical }{\mathcal{I}}			

\newcommand{\Acalvp}{\mathcal{A}_{\mbox{\scriptsize$\mathfrak{f}$}}}
\newcommand{\Icalvp}{\mathcal{I}_{\mbox{\scriptsize$\mathfrak{f}$}}}

\newcommand{ \Lcal }{{L}}

\newcommand{ \Dcal }{\mathcal{D}}			
			
\newcommand{ \Bcal }{\mathcal{B}}

\newcommand{ \C }{\mathbb{C}}				
\newcommand{ \F }{\mathbb{F}}				
\newcommand{ \Cnn}{\C^{n\times n}}
\newcommand{ \eR }{\overline{\mathbb{R}}}		
\newcommand{ \R }{\mathbb{R}}			 	
\newcommand{ \B }{\mathbb{B}}			 	
\newcommand{ \N }{\mathbb{N}}				
\newcommand{ \E }{\mathbb{L}}				

\newcommand{\LBB}{\mathbb{L}}

\newcommand{ \dom }{\mathrm{dom}\,}			
\newcommand{ \epi }[1]{\mathrm{epi}({#1})}		
\newcommand{ \cone }{\mathrm{cone}\,}			

\newcommand{ \interior }{ \mathrm{int}\,}		
\newcommand{ \conv }{\mathrm{conv}\,}		
\newcommand{ \rspan }{\mathrm{rspan}\,}		

\newcommand{ \tr }{\mathrm{tr}\,}  			
\newcommand{ \Diag}{\mathrm{Diag}\,}		

\newcommand{ \ip }[2]{\left\langle{{#1},{#2}}\right\rangle} 	
\newcommand{\cip}[2]{\ip{#1}{#2}^{\!\scriptstyle{\mathfrak{c} }}}	
\newcommand{ \Realpart }{\mathrm{Re}}					
\newcommand{ \abs}[1]{\left\vert{#1}\right\vert}				
\newcommand{ \norm }[1]{\left\Vert{#1}\right\Vert}			

\newcommand{ \sd }{\partial}			
\newcommand{ \rsd }{\hat\partial}		
\newcommand{ \dee }{ \mathrm{d} }  		

\newcommand{\sff}{{\mathbf{f}}}	
\newcommand{\sfr}{{\mathbf{r}}}	
\newcommand{\sfi}{\mathsf{i}}			

\newcommand{ \til }[1]{ \tilde{#1}}

\newcommand{\vn}{\vskip0.1truein \noindent}

\newcommand{\eps}{\varepsilon}
\newcommand{\lam}{\lambda}
\newcommand{\Lam}{\Lambda}
\newcommand{\lamo}{{\lam_0}}
\newcommand{\om}{\omega}		
\newcommand{\ka}{\kappa}			
\newcommand{\gam}{\gamma}		
\newcommand{\vp}{ {\mbox{\large$\mathfrak{f}$} }} 		
\newcommand{\sig}{\sigma}		
\newcommand{\Tau}{\mathcal{T}}

\newcommand{\ds}{\displaystyle}
\newcommand{\ssub}[1]{{\scriptscriptstyle{#1}}}
\newcommand{\set}[2]{\left\{#1\, \left\vert\, #2\right.\right\}}

\DeclareMathOperator*{\toarrow}{\longrightarrow}

\newcommand{\map}[3]{#1\,:\, #2\rightarrow #3}

\newcommand{\wt}[1]{{\widetilde{#1}}}
\newcommand{\wh}[1]{{\widehat{#1}}}

\newcommand{\tlam}{\tilde{\lam}}
\newcommand{\adj}{\star}
\newcommand{\simplex}{\bigtriangleup}
\newcommand{\no}{n_0} 	 
\newcommand{\tiln}{\tilde n}

\newcommand{\Gee}{G}
\newcommand{\tilp}{{\tilde p}}

\newcommand{\tXi}{\wt\Xi}
\newcommand{\tX}{{\wt{X}}}

\newcommand{\wc}[1]{\check{#1}}

\title{Variational Analysis of Convexly Generated Spectral Max Functions}
\date{August 15, 2015}
\author{James V.~Burke}\thanks{The work of James V.~Burke is partially supported by NSF grant DMS-1514459.}
\address{University of Washington, Department of Mathematics, Box 354350
 Seattle, Washington 98195-4350  (jvburke@uw.edu)}
\author{Julia Eaton}
\address{University of Washington Tacoma, School of Interdisciplinary Arts and Sciences, Box 358436, Tacoma, WA 98402 (jreaton@uw.edu)}

\begin{document}
\maketitle
\begin{abstract}\noindent
The spectral abscissa is the largest real part of an eigenvalue of a matrix and the spectral radius is the largest modulus. 
Both are examples of spectral max functions---the maximum of a real-valued function over the spectrum of a matrix.
These mappings arise in the control and stabilization of dynamical systems.
In 2001, Burke and Overton characterized the regular subdifferential of the 
spectral abscissa and showed that the spectral abscissa is subdifferentially regular in the sense of Clarke when all active eigenvalues are nonderogatory.
In this paper we develop new techniques to obtain these results for the more general class of convexly generated spectral max functions.
In particular, we extend the Burke-Overton subdifferential regularity result to this class.
These techniques allow us to obtain new variational results for the spectral radius.  
\medskip

\noindent
Dedicated to R. Tyrrell Rockafellar on the occasion of his 80th birthday. 
We gratefully acknowledge Terry's leadership, guidance and friendship.

\medskip
\noindent
{\bf Key words.} spectral function -- spectral max functions -- spectral radius -- spectral abscissa -- subdifferential regularity -- variational analysis -- chain rule
\end{abstract}

\section{Introduction}
\label{intro}
\noindent
A spectral function maps the space of complex $n$-by-$n$ matrices, $\Cnn$, to the extended reals $\eR:=\R\cup\{+\infty\}$
and depends only on the eigenvalues of its argument independent of permutations \cite{burke-overton:01b,lewis:99a}.
Given $f:\C \to \eR$,
the  {\it spectral max function $\vp: \Cnn \to \eR$ generated by $f$} is 
	\begin{eqnarray}
	\label{eqn: vp defn - original}
		\vp(X) 
		:= \max\{ f(\lam) \, | \,  \lam \in \C  \mbox{ and  } \det(\lam I - X)=0 \}.
	\end{eqnarray}
Spectral max functions are continuous if the underlying function $f$ is continuous; however, they are generally not Lipschitz continuous. 
Two important spectral max functions are the spectral abscissa and spectral radius, obtained
by setting 
	\( f(\cdot) = \Realpart(\cdot) \) and  \( f(\cdot) = \abs{\cdot} , \)
 respectively.
These functions are  connected to classical notions of stability for continuous 
and discrete dynamical systems. In recent years, much research has been 
dedicated to 
understanding the variational properties of the spectral abscissa 
\cite{burke-henrion-lewis-overton:06a,burke-lewis-overton:00a,burke-lewis-overton:01a,burke-lewis-overton:02a,burke-lewis-overton:03a,burke-overton:94a,burke-overton:01b,lewis:99a,lewis:03a}.
Burke and Overton \cite{burke-overton:01b} develop a formula for the regular subdifferential of the spectral abscissa and establish its subdifferential regularity 
(see \S\ref{sec: variational analysis background})
on the set of nonderogatory matrices.
Subdifferential regularity allows one to exploit a powerful subdifferential calculus
for such functions in order to describe their underlying variational properties. 
A primary goal of this paper is to extend the Burke-Overton subdifferential regularity result to 
the class of convexly generated spectral max functions defined below (see Theorem~\ref{main result}).

The variational analysis of spectral max functions builds on that
for \textit{polynomial root max functions} developed in
\cite{burke-lewis-overton:04a,burke-lewis-overton:05a,burke-overton:01a}.
Let $\Pcal{n}$ denote the set of complex polynomials of
degree $n$ or less in a single variable over $\C$,
and let $\Pcal{[n]}$ be those polynomials of degree precisely $n$. 
The polynomial root max function generated by $f$ is the mapping
	$\sff: \Pcal{n} \to \eR$ 
defined by 	
	\begin{equation}
	\label{defn:prmf}
	  \sff(p) := \max\{ f(\lam) \, | \, \lam \in \C \mbox{ and }  p(\lam) = 0 \} .  
	 \end{equation}
We say that $\sff$ is \textit{convexly generated} if
the generating function $f$ is proper, convex, and lsc, which we assume throughout.

The results in \cite{burke-lewis-overton:04a,burke-lewis-overton:05a,burke-overton:01a}
are extended in \cite{burke-eaton:12},
where, in particular, the subdifferential regularity of convexly generated polynomial root max functions
is established.
Subdifferential formulas are given when
$f$ satisfies one of the following two conditions
(introduced in \cite{burke-lewis-overton:05a}) at all ``active roots" (see \S\ref{sec: active roots}):
		\begin{gather}
		\label{f-twice-diffable-and-psd}
			\begin{array}{c}
			 \mbox{``$f$ is quadratic, or }
			\mbox{$f$ is $\mathcal{C}^2$ and positive definite at $\lam$"}			  
			\end{array}
			\\
			\label{rspan is C}
			``\rspan(\sd f(\lam)) = \C"
		\end{gather}
where $f'' (\lam;\cdot,\cdot)$ is defined in \eqref{real sense derivatives},
$\sd f(\lam)$ is the usual subdifferential from convex analysis,
and, for $S \subset \C$, 
	\(  \rspan(S) := \set{\tau \zeta}{ \tau \in \R, \ \zeta \in S}  \) 
is the 
\emph{real linear span} of the set $S$.
Observe that condition \eqref{rspan is C} can only hold when $f$ is not  differentiable.
Consequently, \eqref{f-twice-diffable-and-psd} and \eqref{rspan is C} are mutually exclusive.

Following \cite{lewis:99a}, the authors of \cite{burke-overton:01b} directly derive
necessary conditions on the matrix entries of a subgradient for the 
spectral max function (see Theorems~\ref{theorems A B} and ~\ref{thm: section 6}) 
using the permutation invariance of spectral max functions and the Arnold form (see \S\ref{sec: arnold}).
In the case of the spectral abscissa, \cite{burke-overton:01b} establishes a formula for the regular subdifferential using a chain rule applied to the representation
	\( \vp = \sff \circ \Phi_n\),
where \( f(\cdot) = \Realpart(\cdot) \) and
	\( \map{\Phi_n}{\Cnn}{\Pcal{[n]}} \) 
is the  characteristic polynomial map
	\begin{equation}
	\label{eqn: Phi defn}
	 \Phi_n(X) (\lam):= \det(\lam I - X).    
	 \end{equation}
However, since the mapping $\Phi_n$ contains no information on Jordan structure,
this approach does not provide a direct path
to showing that the spectral abscissa inherits the subdifferential regularity of the 
polynomial abscissa in the nonderogatory case.
Instead, subdifferential regularity is established using a non-constructive argument and specialized matrix techniques.
This approach applies to affinely generated spectral max functions, but
all efforts to extend them to general convexly generated spectral max functions have been persistently rebuffed.  
In this paper, we show that an approach based on a more refined use of Arnold form developed in \cite[Theorems 2.2, 2.7]{burke-lewis-overton:01a} can succeed when all active eigenvalues are nonderogatory.

Two approaches for extending the variational theory to convexly generated spectral max functions are presented. 
The first uses the Arnold form (see Theorem~\ref{thm: Arnold Form})
to develop a framework for the application of a standard chain rule \cite[Theorem 10.6]{rw:98a} to simultaneously establish both a formula for the 
subdifferential and subdifferential regularity when all active eigenvalues are nonderogatory (see Theorem~\ref{thm:vp is subdifferentially regular}). 
The second extends the methods in \cite{burke-overton:01b} to develop a formula for the set of regular subdifferentials without a nonderogatory assumption (see Theorem~\ref{thm:vp is subdifferentially regular}). 
However, we again emphasize that the second approach does not provide a path toward establishing subdifferential regularity when $f$ is not affine.

The method of proof for the first approach builds on tools developed in \cite{burke-lewis-overton:01a,burke-overton:01b}.
For $\tiln\le n$, we exploit the fact that a polynomial $\tilp\in\Pcal{\tiln}$ has a local factorization in $\Pcal{\tiln}$ based at its roots (see \eqref{eqn: F sub p defn}). 
This gives rise to what we call the factorization space $\Scal_{\tilp}$
and an associated diffeomorphism $F_\tilp:\Scal_{\tilp} \to \Pcal{\tiln}$ (see \S\ref{factorization space section}).
The key new ingredient is 
our mapping $\map{G}{\C \times \Cnn}{\Scal_{\tilp}}$ (see \eqref{eqn: big G defn}), which takes a matrix $\tX\in\Cnn$ to the ``active factor" (of degree $\tiln \leq n$) 
associated with its characteristic polynomial $\Phi_n(\tX)$ (see \eqref{active roots} and \eqref{active factor defn}).
The resulting framework is described by the following diagram.
	\begin{equation}
	\label{diagram}
	\xymatrix{
	  \C^{n\times n}  \ar[d]_{\Gee(0,\cdot)} \ar[r]^{\mathfrak{f}} & \eR  \\ 
	  \Scal_{\tilp} \ar[r]_{F_{\tilp}}     & \Pcal{[\tiln]}  \ar[u]^{\sff} }
	\end{equation}
The majority of the paper is devoted to developing the variational properties of the mappings described in this diagram.

Numerous issues need to be addressed in order to place previous work into a comprehensive framework within which modern techniques of nonsmooth analysis apply. 
The new framework allows us to extend our knowledge of the variational properties of the spectral abscissa to convexly generated spectral max functions.
Of the several technical refinements/advances needed for the development of this framework, we describe
three of particular significance.
First, specialized notions of differentiability are required since the fields on the domain and range of $\vp$ differ (see \S\ref{Derivatives}).
Second, inner products are developed that exploit the local geometry of the spaces  
	$\Cnn$, $\Pcal{\tiln}$, and $\Scal_{\tilp}$.
These inner products give convenient representations for 
the derivatives of the mappings $G$ and $F_\tilp$ and their adjoints (see \S\ref{inner products}).
Finally, and most importantly, since $\vp$ locally depends only on those eigenvalues that are ``active" (see \S\ref{sec: nonderogatory case}), the mapping $G$ is introduced to
focus the analysis on the eigenspace determined by the active eigenvalues alone, 
thereby preserving the variational structure 
required for the application of a standard nonsmooth chain rule. 

The paper is organized as follows. 
In Section~\ref{background}, 
we build notation and review the necessary background.
In Section~\ref{sec: background polynomial structures} 
we review the polynomial results from 
\cite{burke-eaton:12} and  \cite{burke-overton:01a}.
In Section~\ref{sec: derivative of characteristic factor} we recall  
the Arnold form and use it to develop a representation for the derivative of the 
factors of the characteristic polynomial corresponding to nonderogatory eigenvalues.
In Section~\ref{sec: nonderogatory case}, we derive formulas for the subdifferential and horizon subdifferential of 
	$\vp$ 
at 
	$\wt X$ 
when 
	$\wt X$ 
has nonderogatory active eigenvalues and show the subdifferential regularity of 
	$\vp$ at $\wt X$. 
In Section~\ref{sec: general case}, we derive the regular subdifferential of $\vp$ 
at an arbitrary $\wt X \in \Cnn$ 
and show that $\vp$ is subdifferentially
regular at $\wt X$ if and only if the active eigenvalues of $\wt X$ are nonderogatory.
In Section~\ref{sec: radius} we apply these results to the spectral radius and illustrate them with two examples.


\begin{center}
\begin{tabular}{ll  }
{\bf Section}		& {\bf Notation}    		 			  		 \\ 
\ref{intro}			& $\eR$, \  $f$, \ $\sff$, \ $\vp$, \ $\mathrm{rspan}$, \ $\Phi_n$, $\Pcal{n}$, \ $\Pcal{[n]}$  \\
\ref{inner products}   & $(\LBB, \F, \ip{\cdot}{\cdot}),\ \cip{\cdot}{\cdot}$, \ 
$L^\adj$, \ $\cip{\cdot}{\cdot}_{\C},\ \cip{\cdot}{\cdot}_{\Cnn}$ \\
\ref{Derivatives}	& $\F$-differentiable, \ $h'(x)$, \ $\nabla h(x)$, \ $f'(\zeta;\delta)$, \  $f''(\zeta;\delta, \delta)$	 \\
\ref{sec: variational analysis background}	
				&	$\R_{++}$, \ $\R_+$, \ $S^\infty$, \ $\dee h(x)$, \ $\rsd h(x)$,  \ $\sd h(x)$, \ $\sd^\infty h(x)$, \ $h'(x;v)$ \\
\ref{factorization space section} 
				& $\Mcal{\tiln}$, \ $\elem{\ell}{\lamo}$, \ $\Scal_{\tilp}$, \ $F_{\tilp}$, \ $\tau_{(k,\lamo)}$, \ $\Tau_{\tilp}$, \ $\ip{\cdot}{\cdot}_{\Scal_\tilp}$, \ $\ip{\cdot}{\cdot}_{(\Pcal{\tiln},\tilp)}$ \\
\ref{sec: active roots} 
				& $\Acal_{\sff}(p)$, \ $\Ical_{\sff}(p)$ \\
\ref{sec: arnold} 	& $\tXi$, \ $\wt P$, \ $\wt B$, \ $J_j$, \ $J_j^{(k)}$, \ $m_{jk}$,  \ $I_{m_{jk}}$, \ $N_{jk}$, \ $q_j$,
					 \ $m_j$, \ $\tiln$,\ $\no$, \ $\lam_{js}(X)$  \\ 
\ref{subsec:char poly} 
				& $g_j(X)$,  \ $\widehat{J}_j(X)$,\  $\Gee(\zeta,X)$, \ $R(v)$
\end{tabular} 
\end{center}

\section{Background}
\label{background}
We review 
	inner product spaces,
	differentiability with respect to a field,
	and notation and definitions from variational analysis.

\subsection{Inner products}
\label{inner products}
A linear space $\LBB$ over the field $\F$ with real or complex inner product $\ip{\cdot}{\cdot}$ is called
an inner product space, denoted by the triple $(\LBB, \F, \ip{\cdot}{\cdot})$.
Only the complex $\C$ and real $\R$ fields are considered.
When it is important to emphasize that the inner product is complex, we write  $\cip{\cdot}{\cdot}$.
A complex inner product space $(\LBB, \C, \cip{\cdot}{\cdot})$
can also be viewed as a real inner product space $(\LBB, \R, \ip{\cdot}{\cdot})$   
with real inner product
	$\ip{\cdot}{\cdot} := \Realpart(\cip{\cdot}{\cdot})$.
This distinction is important when considering linear transformations and their adjoints, 
since both depend upon the choice of the underlying field. 
Recall that if $\dim(\LBB, \C, \cip{\cdot}{\cdot})=n$, then $\dim(\LBB, \R, \ip{\cdot}{\cdot})=2n$.

Let $\F_0, \F_1, \F_2$ be the fields $\C$ or $\R$ with $\F_0$ a subfield of both $\F_1$ and $\F_2$.
For $i=1,2$, let
	$(\LBB_i, \F_i, \ip{\cdot}{\cdot}_i )$,  
be finite-dimensional inner product spaces.
Note that both $\LBB_1$ and $\LBB_2$ are inner product spaces over $\F_0$ with the appropriate 
$\F_0$ inner product.
A transformation  
	$\Lcal: \LBB_1\to \LBB_2$
is said to be $\F_0$-linear  if for all 
	$\alpha, \beta \in \F_0$, 
	$x,y\in\LBB_1$,
we have 
	\( \Lcal(\alpha x + \beta y) = \alpha \Lcal(x) + \beta \Lcal(y).  \)
The adjoint mapping $\Lcal^\adj: \LBB_2\to \LBB_1$ 
\cite{Lax} is the unique $\F_0$-linear transformation satisfying
	\begin{equation}
	\label{adjoint defn}
		\mathrm{Re}(\ip{x}{\Lcal y}_2 )= \mathrm{Re}(\ip{\Lcal^\adj x}{y}_1) \quad \forall \ (y,x) \in \LBB_1 \times \LBB_2.
	\end{equation}
As given, this definition appears insufficient since the $\F_0$ inner product is not specified in \eqref{adjoint defn};
however, this is resolved by noting that the imaginary part of a complex inner product 
can be obtained from knowledge of only its real part 
since 
	\( \mathrm{Re}(\cip{\sfi x}{y}) = \mathrm{Im}(\cip{x}{y}) \), 
where $\sfi := \sqrt{-1}$.
That is, the adjoint can be obtained using only the real part of the $\F_0$ inner product.

It is essential to be mindful of the distinction between the 
adjoint of a linear operator and the Hermitian adjoint of a matrix.
The adjoint of a linear operator $\Lcal$ is denoted by $\Lcal^\adj$, whereas the
Hermitian adjoint of a matrix $M\in \C^{n\times m}$ is denoted by $M^*$.
For example, the following elementary lemma is key.

\begin{lemma}[Adjoint of a linear functional]
\label{lem: adjoint operator versus matrix}
Let 
	$(\LBB, \F, \ip{\cdot}{\cdot})$ 
be a finite dimensional inner product space where $\F$ is either $\C$ or $\R$.
Given $y \in \LBB$, define $\Lcal: \LBB \to \F$ as the linear functional
	\( \Lcal x := \ip{y}{x} \)
for all $x \in \LBB$.
Then $\Lcal^\adj \zeta  = \zeta y$, where the left-hand side shows the 
action of the adjoint, and the right-hand side is multiplication of $y$ by the scalar $\zeta \in \F$.
\end{lemma}

The next lemma provides a key tool in our construction of inner products, 
which in turn impacts the nature of the adjoint operator.
\begin{lemma}[Inner product construction]
\cite[Lemma 4.1]{burke-eaton:12}
\label{ips and isomorphisms}
Let $\LBB_1$ and $\LBB_2$ 
be finite dimensional vector spaces over $\F = \C$ or $\R$, 
and let 
	$\map{\Lcal}{\LBB_1}{\LBB_2}$ 
be an $\F$-linear isomorphism.
\begin{enumerate}
\item
Suppose $\LBB_2$ has inner product $\ip{\cdot}{\cdot}_2$.  
Then the bilinear form 
	$\map{\Bcal}{\LBB_1 \times \LBB_1}{\F}$ 
given by
	$\Bcal(x, y) := \ip{\Lcal x}{\Lcal y}_2$ 
for all $x,y\in\LBB_1$ is an inner product on $\LBB_1$.
Denote this inner product by 
	$\ip{\cdot}{\cdot}_1:=\Bcal(\cdot, \cdot)$. 
Then the adjoint 
	$\map{\Lcal^\adj}{\LBB_2}{\LBB_1}$ 
with respect to the inner products 
	$\ip{\cdot}{\cdot}_1$ 
and  
	$\ip{\cdot}{\cdot}_2$ 
equals $\Lcal^{-1}$.
\item
 Let 
 	$(\LBB_i,\F, \ip{\cdot}{\cdot}_i)$  
be inner product spaces for $i=1,2$.
If 
	$\ip{x}{y}_1=\ip{\Lcal x}{\Lcal y}_2$
for all  $x,y\in \LBB_1$, then 
	$\Lcal^\adj=\Lcal^{-1}$ 
with respect to these inner products.
\end{enumerate}
\end{lemma}

As a vector space, we endow $\C$ with the
complex  inner product
	$\cip{\xi}{\zeta}_{\C}:=\bar\xi\zeta$. 	
This inner product defines the standard complex inner product on 
	$\C^{n+1}$
by 
	\( \cip{(a_0, a_1, \dots, a_n)}{(b_0, b_1, \dots,b_n)}_{\C^{n+1}}
		 := \sum_{\ell=0}^n \cip{a_\ell}{b_\ell}_{\C}  \)
for all 
	\( a_\ell, b_\ell \in \C\), $\ell=0, \dots, n$. 
We also  work on the space 
	$\Cnn$ 
of complex 
	$n\times n$ 
matrices with complex inner product
	$\cip{X}{Y}_{\Cnn}:=\tr(X^*Y)$. 
We use Lemma~\ref{ips and isomorphisms} to construct further inner products as needed. 
In each case, there is a complex version
and an associated real version 
obtained by taking the real part of the complex inner product with 
	$\ip{\cdot}{\cdot} := \Realpart(\cip{\cdot}{\cdot})$,
where the bilinearity of the real inner product is considered to be over the scalar field $\R$.

\subsection{Derivatives}
\label{Derivatives}
Again let $\F_0, \F_1, \F_2$ be the fields $\R$ or $\C$ with 
	$\F_0 \subset \F_1 \cap \F_2$.
For $i=1,2$, let
	$(\LBB_i, \F_i, \ip{\cdot}{\cdot}_i )$,  
be finite-dimensional inner product spaces
with associated inner product norms 	
	$ \norm{\cdot}_{i}$. 
Let 
	$h: \LBB_1 \to \LBB_2$. 
We say that $h$ is 
\textit{$\F_0$-differentiable} 
at $x \in \LBB_1$ if there exists an $\F_0$-linear transformation, denoted  
	$h'(x): \LBB_1 \to \LBB_2$, 
such that
	\[  h(y) = h(x) + h'(x)( y -  x) + o(\norm{y-x}_{1}),  \]
where 
	\(  \lim_{y \to x} \norm{  o(\norm{y - x}_{1} ) / \norm{y-x}_{1}  }_{2} = 0. \)
Clearly every $\C$-derivative is an $\R$-derivative, but the converse is false.
If $h: \LBB_1 \to \F_0$, $h'(x)$ defines an linear functional from $\LBB_1$ into $\F_0$ and, 
since $\LBB_1$ is a vector space over $\F_0$, there exists 
an element of $\LBB_1$, denoted by $\nabla h(x)$, satisfying 
	\begin{equation}
	\label{nabla notation}
		h'(x)z = 
		\begin{cases}
		\mathrm{Re}\ip{\nabla h(x)}{z}_{1} 	& \mbox{ if $\F_0=\R$ } \\ 
		\ip{\nabla h(x)}{z}_{1}  			& \mbox{ if $\F_0=\C$} \\ 
		\end{cases}
	\qquad \quad\forall \ z \in \LBB_1.
	\end{equation}
We call $\nabla h(x)$ the gradient of $h$ at $x$, 
and recall from our previous discussion that the vector
$\nabla h(x)$ is the same element of $\LBB_1$ regardless
of whether $\F_0$ equals $\C$ or $\R$. 
Lemma~\ref{lem: adjoint operator versus matrix} tells us that  
	\begin{equation}
	\label{adjoint of derivative}
		 h'(x)^\adj \zeta = \zeta \nabla h(x).
	\end{equation}

Since we make extensive use of $\R$-differentiability for real-valued functions on $\C$,
following \cite{burke-overton:01b},
we show how to construct the $\R$-derivative in this case.
Define the $\R$-linear transformation 
	\( \Theta: \R^2 \rightarrow \C \)  
by
	\( \Theta(x_1, x_2)  :=  x_1  +  \sfi x_2 \).
The inverse mapping  is 
	\( \Theta^{-1}(\zeta) = (\mathrm{Re}(\zeta), \mathrm{Im}(\zeta)). \)
Since 
	$\mathrm{Re}(\ip{\zeta}{\Theta(x_1,x_2)}) = \ip{\Theta^{-1}(\zeta)}{(x_1,x_2)}$
for all 
	$\zeta \in \C$ and $(x_1,x_2)\in\R^2$,
we have  $\Theta^\adj = \Theta^{-1}$.
Given $f: \C \to \R$, define 
	\( \tilde{ f } : \R^2 \rightarrow \R \)
by 
	\( \tilde{ f }  :=   f  \circ  \Theta. \)
If $\tilde f$ is differentiable over $\R$ in the usual sense, then, by the chain rule, $f$ is $\R$-differentiable
with gradient 
	\( \nabla f (\zeta)  =  \Theta \nabla \tilde{ f }(\Theta^{-1}\zeta) \), 
and this is consistent with the notation in
\eqref{nabla notation}.
In \cite{burke-lewis-overton:05a} this  
is called \emph{differentiable in the real sense}.
We say that $f$ is \textit{twice $\R$-differentiable} if 
	\( \tilde{ f } \)
is twice differentiable over $\R$ in the usual sense.  
In this case,
a further application of the chain rule yields 
	\( f(\zeta + \delta) = f(\zeta) +  \ip{\nabla f(\zeta)}{\delta} + (1/2)\ip{\delta}{\nabla^2 f(\zeta)\delta}
	+ o(\abs{\delta}^2)\), 
where 
	\( \nabla^2 f (\zeta)  
	:=  \Theta \nabla ^2 
		\tilde{ f }  (\Theta^{-1}\zeta)\Theta^{-1} . \)
Since we only use $\R$-differentiability of functions $f:\C \to \R$,
we simply say that $f$ is differentiable.
We also make use of the following notation: 
	\begin{equation}
	\label{real sense derivatives}
	  f'  (\zeta; \delta) := \ip{ \nabla f (\zeta)}{\delta} 
	  \quad\mbox{and}\quad
	  f'' (\zeta; \delta, \delta)  := \ip{\delta}{ \nabla^2 f (\zeta)\delta},  
	  \end{equation}
where it follows that 
	\( f'  (\zeta; \delta)  =  \lim_{t \downarrow 0} (\,f (\zeta +  t \delta) - f (\zeta ) \, )/ t .\)
We say that 
	$\nabla^2 f (\zeta)$ 
is positive definite if 
	$\ip{\delta}{ \nabla^2 f (\zeta)\delta}>0$ 
for all $\delta \neq 0$. 
We say $f$ is quadratic if $\nabla^2 f (\zeta)$ is constant in $\zeta$,
and we say that $f$ is $\mathcal{C}^2$ at $\tlam$ if the map $\lam \mapsto \nabla^2 f$ is continuous at $\tlam$. 
For example, the function 
	\( r_2(\zeta) := \abs{\zeta}^2/2 \)
used in \cite[Section 7]{burke-eaton:12} and Section~\ref{sec: radius} below is quadratic with 
	\( \nabla r_2 (\zeta) = \zeta \)
and
	\( \nabla^2 r_2 (\zeta) \) being the identity map on $\C$. 
This notation clarifies the key hypotheses 
\eqref{f-twice-diffable-and-psd} and \eqref{rspan is C}.

\subsection{Variational analysis review}
\label{sec: variational analysis background}
We use the techniques of variational analysis from 
\cite{Cla:83,Mord:05,rw:98a}. 
Let 	
	$(\E, \C, \cip{\cdot}{\cdot})$ 
be a finite-dimensional inner product space 
with associated real inner product
	$\ip{\cdot}{\cdot}:=\Realpart \cip{\cdot}{\cdot}$.
Let $C$ be a non-empty subset of $\E$.
The \textit{tangent cone} to $C$ at a point 
	$x  \in  C$
is  
	\begin{equation}
	\label{eqn: tangent cone defn}
		T_C(x) \!:=\! \{  d \,|\, \exists \{x^\nu\}  \!\subset\!  C, 
		\{ t_\nu \}  \!\subset\!  \R_{++} \text{ s.t. } x^\nu  \!\to\!  x,  
		t_\nu  \!\downarrow\!  0 \text{ and } t^{-1}_\nu(x^\nu \!-\! x)\! \to\! d  \},           
	\end{equation}
where $\R_{++}:=(0,\infty)$.
The tangent cone is a closed subset of $\E$ \cite[Proposition 6.2]{rw:98a}.
The \textit{regular normal cone} to $C$ at a point $x\in C$ is 
	\begin{equation}
	\label{eqn: regular normal cone defn}
		\widehat{N}_C(x) 
		:= \{ z \,|\,  \ip{z}{v}  \leq  0 \text{ for all } v  \in  T_C(x) \}. 
	\end{equation}
Given $S\subset\E$, the \textit{horizon cone} of $S$ is
	\[ S^\infty  :=  
	\set{z  \in  \E}{\exists\,\{x^\nu\}  \subset  S,\,  \{t_\nu\}  \subset \R_+\mbox{ s.t. }
	t_\nu  \downarrow  0\mbox{ and } t_\nu x^\nu\rightarrow z},
	\]
where $\R_+:=[0,\infty)$.
The horizon cone is always a nonempty closed cone. 	
If $S$ is convex, $S^\infty$ is the usual \textit{recession cone} 
from convex analysis.
Let 
	\(h : \E \to \overline{\R}\).
The \textit{essential domain} of $h $ is
        \[ \dom(h ) := \{x  \in  \E \,|\,   h (x)  <  \infty \} . \]
The function $h$ is said to be proper if $\dom(h)\ne\emptyset$.
The \textit{epigraph} of $h$ is given by
	\[ \epi{ h } :=  \{ (x, \beta)  \in  \E  \times  \R \,|\, h (x)  \leq  \beta \}. \]
The \textit{subderivative} of $h$ is the map 
	\(\dee h (x): \E \to \overline{\R}\cup\{-\infty\}\)
given by
	\begin{equation*}
	\label{eqn: subderivative defn}
		\dee h (x)(\tilde {v}) := 
		\liminf_{t \downarrow 0, v \to \tilde{v}} \, ( \,  h (x +  t v) - h (x) \, )/ t .
        \end{equation*}
By \cite[Theorem 8.2]{rw:98a}, the subderivative and the tangent
cone to the epigraph 
are related by  
	\(\epi{\dee h (x)}=T_{\epi{h}}(x,h(x))\)
	\(\forall\ x\in\dom(h)\).
The \textit{regular subdifferential} of $h$ at $x \in \dom(h)$ is the set of
\textit{regular subgradients}:
	\[  \rsd h(x)  :=  \{ z \,|\, h(y)  \geq  h(x)  +  \ip{z}{y - x}  +  o(\norm{y - x}) 
	\quad  \forall\,y\in \E\}. \]
The regular subdifferential is a closed, convex subset of $\E$.
By \cite[Theorem 8.9]{rw:98a},
the regular subdifferential and the regular normal cone to the epigraph are related by
	\( \rsd h(x) = \{ {z} \ | \ {(z,-1)\in \widehat N_{\epi{h}}(x,h(x))} \}. \)
The general (or limiting) subdifferential of $h$ at $x$
is given by
	\begin{equation}
	\label{general-sd-formula}
		\sd h (x) :=  
		\left\{
			 z  \left\vert
			\begin{array}{c}
			\exists \,  \{(x^\nu,z^\nu)\}  \subset \dom(h)\times \E, \mbox{ with }\,z^\nu \in \rsd h( x^\nu)\, \forall\, \nu,
			\\  
		 	  (x^\nu,z^\nu) \to (x,z), \mbox{ and }h(x^\nu)\to h(x)  
			\end{array}
			\right.
		\right\},
	\end{equation}
and the horizon subdifferential to $h$ at $x$ is given by
	\begin{equation*}
		\label{horizon-subdifferential}
		\sd^\infty h(x) \!:= \!
		\left \{\! z \left \vert\!
		\begin{array}{c}
	\exists \,  \{(x^\nu,z^\nu,\beta_\nu)\} \! \subset\! \dom(h)\!\times\! \E\!\times\! \R_{++}, 
	\!\mbox{ with }\! z^\nu \!\in\! \rsd h( x^\nu)\, \forall\, \nu,
	\\ 
	\beta_\nu\downarrow 0,\  (x^\nu,\beta_\nu z^\nu) \to (x,z), \mbox{ and }h(x^\nu)\to h(x)
		\end{array}
		\right.\!\!\right \}\!.
	\end{equation*}
The function $h$ is said to be \textit{subdifferentially regular} at $x$ if 
	\( \sd h(x) = \rsd h(x)\)
and 
	\( \sd^\infty h(x) = \rsd h(x)^\infty .\)
Subdifferential regularity 
allows the
development of a rich subdifferential calculus. 

If $h$ is differentiable at $x$, then, by \cite[Exercise 8.8]{rw:98a},
	\[ \rsd h(x) =\sd h(x) = \{\nabla h(  x)\} 
	\quad\mbox{and}\quad
	 \dee h(x)(v) = \ip{\nabla h(x)}{v}. \]
If $h$ is proper, convex, and lsc, then, by \cite[Proposition 8.12]{rw:98a}, 
$h$ is subdifferentially regular at every $x \in \dom(h)$,
in which case 
	\begin{equation}
	\label{subderivative is directional}
		\dee h (x)(v) = h'(x;v) :=
		\lim_{t \downarrow 0}   ( \,  h (x +  t v) - h (x) \, )/ t. 
        \end{equation}
Consequently, in this case, the subderivative corresponds to the usual notion of a directional derivative
\cite[page 257]{rw:98a}.
This notation is consistent with the usage in \eqref{real sense derivatives}, where the function is assumed to be $\R$-differentiable.

\section{Polynomials}
\label{sec: background polynomial structures}

The focus of study in 
\cite{burke-eaton:12,burke-lewis-overton:05a,burke-overton:01a} 
is polynomial root functions,  
a special class of which  is polynomial root max functions 
\eqref{defn:prmf}.
The goal is to apply the variational results for $\sff$ in
\cite{burke-eaton:12,burke-lewis-overton:04a,burke-lewis-overton:05a,burke-overton:01a} to  
	\(  \vp \) 
when 
	$\sff$ 
is convexly generated.
These functions depend on what will be called the active factor (see \S\ref{sec: active roots} and \eqref{active factor defn}). 
In the discussion of the application of the polynomial results to the characteristic polynomial, 
the integer $\tiln \in \N$ is used to denote either the degree of the active factor of the characteristic polynomial 
(the nonderogatory case, \S\ref{sec: nonderogatory case}) 
or the degree of the characteristic polynomial itself (the derogatory case, \S\ref{sec: general case}).

\subsection{Factorization spaces and their inner products on $\Pcal{\tiln}$ \cite{burke-overton:01a}}
\label{factorization space section}

Let $\preceq$ denote the \textit{lexicographical order} on 
$\C$,  where, for 
	\( z_s := x_s +  \sfi y_s,\ x_s,y_s  \in  \R, \) \( s = 1,2,\) 
we have
	\(z_1  \preceq  z_2\) 
if and only if either 
	\(x_1  <  x_2\)
 or 
 	(\(x_1  =  x_2\) 
and 
	\(y_1  \leq  y_2\)).
Let
	$\Mcal{\tiln}\subsetneq \Pcal{[\tiln]}$ 
be the affine set of \emph{monic} polynomials of degree $n$ 
with
	$\Mcal{0} := \{1\}$.
Given 
	$\tilp \in \Mcal{\tiln}$, 
write
	\begin{equation}
	\label{p-factor}
		\tilp:=
		\mbox{$\prod_{j=1}^m$} \elem{n_j}{{\tlam}_j} ,
	\end{equation}
where 
	${\tlam}_1, \dots, {\tlam}_m$
are the distinct roots of $\tilp$, ordered lexicographically with multiplicities 
	$n_1, \dots, n_m$, 
and the monomials 
	\( \elem{\ell}{{\tlam}_j}\in\Pcal{\tiln} \)
are defined by 
	\begin{eqnarray}
	\label{eqn: elementary polynomial}
		\elem{\ell}{\lamo}(\lam) := (\lam - \lamo)^\ell ,\quad\mbox{for \( \ell=0, \dots, \tiln ,\)
		and all $\lamo\in\C$.}
	\end{eqnarray}
Note that, for each fixed
	\(  \lamo \in\C , \) 
the monomials \eqref{eqn: elementary polynomial} form a basis for the linear space 
	\( \Pcal{\tiln} .\)
The factorization space $\Scal_{\tilp}$ for $\tilp$ is given by
	\begin{align*}
		\Scal_{\tilp }& :=
		\C \times \Pcal{n_1-1} \times \Pcal{n_2-1} \times \cdots \times \Pcal{n_m-1},
	\end{align*}
where the component indexing for elements of $\Scal_{\tilp}$ starts at zero so that the
$j$th component is an element of $\Pcal{n_j-1}$. 
The spaces 
	$\Pcal{\tiln}$ 
and 
	$\Scal_{\tilp}$ 
are related through the mapping
	\( F_{\tilp}: \Scal_{\tilp} \to \Pcal{\tiln} \)  
given by
	\begin{equation}
	\label{eqn: F sub p defn}
	\begin{aligned}
		F_{\tilp}(q_0, q_1, q_2, \dots, q_m) 
			&:= (1 + q_0) 
			\mbox{$\prod_{j=1}^m$} ( \elem{n_j}{{\tlam}_j} + q_j).
	\end{aligned}
	\end{equation}
We have
	\( F_{\tilp}(0)  =  \prod_{j=1}^m \elem{n_j}{{\tlam}_j}=\tilp, \)
and, since the factors in \eqref{p-factor}
are relatively prime, there exist neighborhoods
$U$ of 0 in $\Scal_{\tilp}$ and $V$ of 
	\( {\tilp} \) 
in 
	$\Pcal{\tiln}$ 
such that 
	\( F_{\tilp} \vert_U : U \to V \) 
is a diffeomorphism (over $\C$)  \cite[Lemma 1.4]{burke-overton:01a}.  
The $\C$-derivative
	\( F_{\tilp}'(0): \Scal_{\tilp} \to \Pcal{\tiln} \)
is invertible and given by
	\begin{equation}
	\label{eqn: F nabla}
		 F'_{\tilp}(0)(\om_0, w_1, w_2, \dots, w_m) 
			= \om_0 {\tilp}
			+ \mbox{$\sum_{j=1}^m$ } r_j w_j, 
			\quad
			\mbox{where \(  r_j := {\tilp}/\elem{n_j}{{\tlam}_j} \).}
	\end{equation}

For each  
	\( \lamo\in\C \),
define the scalar \textit{Taylor maps} 
	$\map{\tau_{(k,\lamo)}}{\Pcal{\tiln}}{\C}$ 
by 
	\begin{equation}
	\label{eq:comp Taylor}
	\tau_{(k,\lamo)}(q):=q^{(k)}(\lamo)/k!\ ,
	\mbox{ for  }
	k = 0,1,2, \dots, \tiln,
	\end{equation}
where 
	\( q^{(\ell)} \)
is the $\ell$th derivative of $q$.  
Each mapping
	\( \tau_{(k, \lamo)} \)
takes a polynomial to its $k$th Taylor coefficient at 
	\( \lamo . \)
Note that the mapping 
	\( \tilde\tau_k(q,\lam) : =  \tau_{(k,\lam)}(q) \)
is continuous in $q$ and $\lam$ \cite{burke-overton:01a}.
Define the $\C$-linear isomorphism  
	\( \mathcal{T}_{\tilp}: \Scal_{\tilp} \to \C^{\tiln+1} \)
by
	\begin{align}
	\label{eqn: Tau defn}
		\Tau_{\tilp} (u )  
		& := [\mu_0, 
			(\tau_{(n_1-1, \tlam_1)}(u_1),
			\dots,
			\tau_{(0, \tlam_1)}(u_1)),
			\dots, \\ \nonumber& \hspace{3.0cm}
			(\tau_{(n_m-1, \tlam_m)}(u_m),
			\dots,
			\tau_{(0, \tlam_m)}(u_m))]^T
	\\ \nonumber
		& = [\mu_0,(\mu_{11},\dots,\mu_{1n_j}),\dots,(\mu_{m1},\dots,\mu_{mn_m})]^T, 
	\end{align}
where
	\begin{align}
	\label{u-defn}
		u := ( \mu_0, u_1, u_2, \dots, u_m)\in \Scal_{\tilp}, \ 
		& u_j := 
		\mbox{$\sum_{s=1}^{n_j}$} \mu_{js} \elem{n_j-s}{\tlam_j} ,  	 
		\mbox{ and }  	 \mu_0,   \mu_{js} \in \C, 
	\end{align}
for all 
	\( s \in\{ 1, 2, \dots n_j \}\)
and
	\( j \in\{ 1, 2, \dots m \}. \)

By Lemma~\ref{ips and isomorphisms},
	\( \Tau_{\tilp} \)
induces an inner product on 
	\( \Scal_{\tilp} \) 
by setting 
	\begin{equation}
	\label{eq:ip on Sp}
		\cip{u}{w}_{\Scal_{\tilp}} := \cip{\Tau_{\tilp}(u)}{\Tau_{\tilp}(w)}_{\C^{\tiln+1}}, 
	\mbox{  for all  }
	 u,w  \in  \Scal_{\tilp} , 
\end{equation}
where $u$ is given in \eqref{u-defn},  
	\begin{equation}
	\label{w-defn}
	\begin{aligned}
		w  & =  (\omega_0, w_1, w_2, \dots, w_m), \quad\mbox{with} \\
		w_j & = \mbox{$\sum_{s=1}^{n_j}$} \omega_{js} \elem{n_j-s}{\tlam_j} \quad\mbox{for  $j\in\{1, \dots, m\}$ \ \ and } \\
		\omega_0,\ &  \omega_{js} \in \C \quad\mbox{for $s\in\{1,\dots, n_j\}$, $j\in\{1, \dots, m\}$.}  
	\end{aligned}
	\end{equation}
Moreover, again by Lemma~\ref{ips and isomorphisms}, 
	\[ \Tau_{\tilp}^{\adj}=\Tau_{\tilp}^{-1}\]
with respect to the  inner products 
	$\cip{\cdot}{\cdot}_{\Scal_{\tilp}}$ 
and  
	$\cip{\cdot}{\cdot}_{\C^{\tiln+1}}$.

We are now ready to construct an inner product on 
	$\Pcal{\tiln}$ 
relative to $\tilp$.
Recall that
	\( F'_{\tilp}(0): \Scal_{\tilp} \to \Pcal{\tiln} \)
given in \eqref{eqn: F nabla} is a $\C$-linear isomorphism.
So for every
	\(  z, v  \in  \Pcal{\tiln}, \)
there exists 
	\( u, w \in  \Scal_{\tilp} , \) 
having representations \eqref{u-defn}  and  \eqref{w-defn}, such that
	\begin{eqnarray}
	\label{z-defn}
		z = F'_{\tilp}(0)u  
		\quad \mbox{and}\quad 
		v = F'_{\tilp}(0)w. 
	 \end{eqnarray}
By Lemma~\ref{ips and isomorphisms},
	\( F'_{\tilp}(0)^{-1} \)
induces an inner product $\cip{\cdot}{\cdot}_{(\Pcal{\tiln},{\tilp})}$ on 
	\( \Pcal{\tiln} \) 
based on the inner product 
	\( \cip{\cdot}{\cdot}_{\Scal_{\tilp}} \)
by setting
	\begin{align}
	\label{F ip}
		\cip{z}{v}_{(\Pcal{\tiln},{\tilp})}	 
		& := \cip{F'_{\tilp}(0)^{-1} z}{F'_{\tilp}(0)^{-1} v }_{\Scal_{\tilp}}   
		  = \cip{u}{w}_{\Scal_{\tilp}}  \\ 	\notag
		& = \cip{ (\mu_0, u_1, u_2, \dots, u_m)}{ (\omega_0, w_1, w_2, \dots, w_m)}_{\Scal_{\tilp}}  \\ 	 	
	\label{inner product}
		& =  \bar \mu_0 \omega_0  + 
		\mbox{$\sum_{j=1}^m \sum_{s=1}^{n_j} $}
		 \bar \mu_{js}\omega_{js} ,
	\end{align}
where $z$ and $v$ are given in \eqref{z-defn}, and 
$u$ and $w$ are given in \eqref{u-defn} and \eqref{w-defn}, respectively.
With respect to these inner products, \eqref{F ip} (or Lemma~\ref{ips and isomorphisms}) gives
	\begin{equation}
	\label{F' inverse is F' adjoint}
	(F'_{\tilp}(0))^\adj  =  F'_{\tilp}(0)^{-1}.  
	\end{equation}

\subsection{Active roots and active indices}
\label{sec: active roots}
Let
$p\in\Pcal{[\tiln]}$ and 
denote by $\Xi_p:=\{\lam_1, \dots, \lam_m\}$ the
distinct roots of $p$, ordered lexicographically.
The set of \emph{active roots} of $\sff$ at $p$ is given by
	\begin{equation}
	\label{active roots}
	 	\Acal_{\sff}(p) :=  \set{ \lam_j \in \Xi_p}{f(\lam_j) = \sff(p)}.
	\end{equation}
If $\lam_j \in \Xi_p \setminus\Acal_{\sff}(p)$, then
 $f(\lam_j)<\sff(p)$, and we say that $\lam_j$ is \emph{inactive.}
The set of \emph{active indices} of $\sff$ at $p$ is given by 
	\[  \Ical_{\sff}(p):= \set{j \in \{1, \dots, m\}}{\lam_j \in \Acal_{\sff}(p)}.  \]

\subsection{The subdifferential and subderivative for polynomial root max functions}
Again assume that 
	$f: \C \to \eR$ 
is proper, convex, and lsc. 
As discussed at the end of Section~\ref{sec: variational analysis background},
$f$ is subdifferentially regular and so \eqref{subderivative is directional} holds.

\begin{theorem}
\cite[Proposition 5.5, Theorem 5.3 and Theorem 6.2]{burke-eaton:12}
\label{Thm: Burke-Eaton 12}
Let
	$\tilp \in \Pcal{\tiln}\cap \dom(\sff)$ 
have degree $\tiln$ with decomposition \eqref{p-factor}. 
Assume that $f$ satisfies either \eqref{f-twice-diffable-and-psd} or \eqref{rspan is C}
at $\tlam_j$ for  $j \in \Ical_\sff(\tilp)$, and  that 
	\begin{equation}
	\label{eq:not zero}
	\sd f(\tlam_j) \neq \{0\}\quad\forall\ j \in \Ical_\sff(\tilp).
	\end{equation}
Given $j \in \Ical_\sff(\tilp)$, define
	\[
	\Qcal(\tlam_j)   :=  -\cone(\sd f(\tlam_j)^2) + \sfi(\rspan(\sd f(\tlam_j)^2)) \ \subset \C, 
	\]
where $\sd f(\tlam_j)^2 := \{g^2  | g \in\sd f(\tlam_j)\}$ and,
for $S \subset \C$, 
	\[ \cone(S) := \{\tau \zeta | \tau \in \R_+,\zeta \in S\}. \] 
For $\tlam_j \in \Acal_\sff(\tilp)$ at which
\eqref{f-twice-diffable-and-psd}  holds, define
	\[   \Dcal(n_j,\tlam_j) \!:= \!	
		 \{ \theta\,   |\,  \langle{\theta},{ (\nabla f(\tlam_j))^2}\rangle_{\C} \!\leq\!  \langle{ \sfi \nabla f (\tlam_j)}, {\nabla^2 f(\tlam_j)
	 	(\sfi \nabla f (\tlam_j))}\rangle_{\C}/n_j \}  \subset \C.   \]
Next, set	
	\[   D_{\tilp}  : =  \conv
		(\{0\}\times \bigtimes_{j=1}^m  \Gamma (n_j, {\tlam}_j) ) \ \subset \ \C^{\tiln+1},	 \] 
where, for 
	$j\notin \Ical_{\sff}(\tilp)$, 
	$\Gamma (n_j, {\tlam}_j):=\{0\} \subset \C^{n_j }$
and, for $j \in \Ical_{\sff}(\tilp)$,
	\begin{alignat*}{1}
		 \Gamma (n_j, {\tlam}_j)  := 
		\begin{cases}
			(- \sd f(\tlam_j)/n_j )		\times \Dcal( n_j,{\tlam}_j) \times \C^{n_j-2} 
	 	&\mbox{if \eqref{f-twice-diffable-and-psd}  holds at $\tlam_j$,} \\  
			(- \sd f(\tlam_j)/n_j )		\times \Qcal({\tlam}_j) \times \C^{n_j-2}  
		&\mbox{if \eqref{rspan is C} holds at $\tlam_j$ }  
		\end{cases}
		\ \subset \C^{n_j}.
	\end{alignat*}
Then
	\[ D_{\tilp}^\infty= 
	 \{0\}\times\bigtimes_{j=1}^m  \Gamma (n_j, {\tlam}_j)^\infty
	\quad\mbox{with}\quad
	 \Gamma (n_j, {\tlam}_j)^\infty\!=\! \{0\}\times \Qcal(\tlam_j)\times \C^{n_j-1},  \]
and, with respect to the inner product 
	$\ip{\cdot}{\cdot}_{(\Pcal{\tiln},{\tilp})}$ 
given in \eqref{inner product}, 
 	\begin{equation}
	\label{polynomial subdifferential}
	\rsd \sff({\tilp}) = F'_{\tilp}(0)\circ\Tau_{\tilp}^{-1} (D_{\tilp} )
	\quad\mbox{and}
	\quad 
	\rsd \sff({\tilp})^\infty = F'_{\tilp}(0)\circ\Tau_{\tilp}^{-1}   (D_{\tilp}^\infty) ,
	\end{equation}  
where $F'_{\tilp}(0)$ and $\Tau_{\tilp}$ are given in \eqref{eqn: F nabla} and \eqref{eqn: Tau defn}, respectively.

If $f$ satisfies  \eqref{f-twice-diffable-and-psd} for all $\tlam_j \in \Acal_\sff(\tilp)$, then
	$\sff$ 
is subdifferentially regular at ${\tilp}$ and \eqref{polynomial subdifferential}
gives the general and horizon subdifferentials, respectively, for $\sff$ at $\tilp$.
\end{theorem}

\begin{theorem}
\label{subderivative general}
\cite[Theorem 3.3]{burke-eaton:12}  
Let $f$ and $\tilp$ satisfy the hypotheses of
Theorem~\ref{Thm: Burke-Eaton 12}.
Let 
	\( v  =  F'_{\tilp}(0)(\om_0, w_1, \dots, w_m), \)
where
	\(
		w_j  \!=\! \sum_{s=1}^{n_j} \om_{js}
		\elem{n_j-s}{{\tlam}_j}
		\, \in \Pcal{n_j-1}
	\)
	for  all
	\( j \in \{1, \dots m\},\)
	with
	\( \om_{js} \in \C,\) for \( s\in\{1, \dots, n_j\}\).
If $v$ 
satisfies
	\begin{align} 
	\label{tc-second}
		0 & = \langle { g}, {  \sqrt{- \om_{j2} }} \rangle_{\C}  && \mbox{ for all }
		j  \in  \Ical_\sff({\tilp})\mbox{ and } g\in \sd f({\tlam}_j),  	\mbox{and} \\	
	\label{tc-third}			
		 0 & =  \om_{js} \quad && \text{ for all } s \in\{ 3, \dots, n_j\} \mbox{ and } j \in \Ical_\sff({\tilp}),\quad
	\end{align}
then
	\[ \dee \sff ({\tilp})(v) 
	= 
	\max_{j \in \Ical_\sff({\tilp})}
	   \{ (\, f ' ({\tlam}_j;- \om_{j1}) + \ka_j\}
	 \]
where, for \( j \in \Ical_\sff({\tilp}) , \)
	\begin{equation*}
	\ka_j :=
	\begin{cases}
	 f '' ({\tlam}_j; \sqrt{- \om_{j2}},\sqrt{- \om_{j2}})  \,)/n_j	 & \mbox{if \eqref{f-twice-diffable-and-psd}  holds at $\tlam_j$,} \\ 
	 0 & \mbox{if \eqref{rspan is C} holds at $\tlam_j$;}
	\end{cases}
	\end{equation*}
	otherwise, $\dee \sff ({\tilp})(v) =+\infty$.
In addition, 
	\( \dee \sff ({\tilp}) \) 
is proper, lsc, and sublinear.
\end{theorem}

\section{The Derivative of the Nonderogatory Factor}
\label{sec: derivative of characteristic factor}
We now recall Arnold form \cite{Arnold:71} and its application
in \cite{burke-lewis-overton:01a}
for describing the local variational behavior of the Jordan form.

\subsection{Arnold Form} 
\label{sec: arnold}
Let  	
 	\begin{equation}
	\label{eqn: distinct evals matrix}
	\tXi :=\{\tlam_1,\dots,\tlam_m\} 
	\end{equation}
be a subset of the distinct eigenvalues of 
	\( \wt X\in\Cnn \).
The Jordan structure of $\wt X$ relative to these eigenvalues is given by 
	\begin{equation}
	\label{eqn: P and J, 1}
	J \!:=\! \wt{P} \wt X \wt{P}^{-1} \!=\!  \Diag (\wt{B}, J_{1},\dots,J_m ),
	 \mbox{ where }  
	J_j \!:=\!\Diag (J_j^{(1)},\dots,J_j^{(q_j)} )
	\end{equation}
and
	\( J_{j}^{(k)} \)
is an
	\( m_{jk}\times  m_{jk}\)
Jordan block 
	\begin{equation}
		\label{Nilpotent block jk}
		J^{(k)}_j := \tlam_j I_{m_{jk}} + N_{jk},
		\quad
		  k=1, \dots, q_j, \  
		  j=1, \dots, m,
	\end{equation}
where  
		\( N_{jk} \in \C^{ m_{jk} \times m_{jk} } \) 
is the nilpotent matrix given by ones on the superdiagonal and zeros elsewhere, and 
	$I_{m_{jk}} \in \C^{m_{jk} \times m_{jk}}$ 
is the identity matrix.
With this notation, 
	\( q_j \) 
is the geometric multiplicity of the eigenvalue 
	\( \tlam_j . \)
The algebraic multiplicity of 
	\( \tlam_j  \)
is 	
	\( n_j  := \sum_{k=1}^{q_j} m_{jk} . \)
 The size of the largest Jordan block for an eigenvalue is
	\begin{equation}
	\label{mj defn}
	m_j := \max_{k=1, \dots, q_j} m_{jk} . 
	\end{equation} 
Set $\tiln := \sum_{j=1}^m n_j$ and $\no := n-\tiln$.
The matrix $\wt{B} \in \C^{\no\times \no}$ (possibly of size $0$) corresponds to the eigenvalues not included in $\tXi$.
The eigenvalue 
	\( \tlam_j \)
is nonderogatory if 
	\( q_j = 1   \) 
(equivalently 
	\( m_j = n_j \)).
The matrix $\wt X$ is nonderogatory if all of its eigenvalues are nonderogatory.

\begin{theorem}
[Arnold Form]
\cite[Theorems 2.2, 2.7]{burke-lewis-overton:01a}
\label{thm: Arnold Form}
Suppose that all eigenvalues in $\tXi$ 
are nonderogatory.
We  suppress the index $k$ since $q_j=1$.
Then there exists a neighborhood 
	$\Omega$ 
of 
	$\wt{X} \in \Cnn$ 
and smooth maps 
	$P:\Omega \to \Cnn$, 
	$B: \Omega \to \C^{\no \times\no}$
and, for   $j \in \{1,\dots,m\}$ and $s\in\{0,1,\dots,n_j-1\}$,  
	$\lam_{js}: \Omega \to \C$ 
such that
		\[ 
		\begin{aligned}
		P(X) X P(X)^{-1} &= \! 
		\Diag\left(B(X),0,\dots,0\right)  +\mbox{$\sum_{j=1}^m \wc{J}_j(X)$} \in \Cnn,   \\
		\lam_{js}(\wt{X}) & =  0, \quad s = 0,1,\dots,n_j-1,\\
		P(\wt{X}) &= \wt{P},  \quad
		B(\wt{X})  =\wt{B}, \quad\mbox{and}\\ 
		P(\wt{X}) \wt{X} P(\wt{X})^{-1} &=  \Diag (B(\wt{X}), J_{1},\dots,J_m ), 
		\end{aligned}
		\]
where 
	\[ \begin{aligned}
	\wc{J}_j(X)
	& := \tlam_j J_{j0} + J_{j1}  
		+ \mbox{$\sum_{s=0}^{n_j-1}$} \lam_{js}(X) J_{js}^*,  
	\\ 
	J_{js}&:=\Diag(0,\dots,0,N_{j}^s,0,\dots,0),   \quad \mbox{and}
	\\
	J_{j0} &:= \Diag(0,\dots,0,I_{n_j},0,\dots,0), 
	\end{aligned} \]
with $N_{j}^s$ and $I_{n_j}$ in the $\tlam_j$ diagonal block (see \eqref{eqn: P and J, 1}).
Finally, the functions $\lam_{js}$ are  uniquely defined on $\Omega$, 
though the maps $P$ and $B$ are not unique.
\end{theorem}

\begin{remark}
Theorem~\ref{thm: Arnold Form} illustrates a fundamental difference between the symmetric and nonsymmetric
cases. In the symmetric case, the matrices are unitarily diagonalizable so there
are no nilpotent matrices $N_j$ and the mappings $\lam_{js}$ reduce to the eigenvalue mapping $\lam_j$. 
In this case, a seminal result due to Lewis \cite[Theorem 6]{lewis:99a} 
shows that the variational properties depend on 
the eigenvalues. 
On the other hand, in the nonsymmetric case they depend on the entire family of functions $\lam_{js}$.
\end{remark}

\begin{lemma}
\label{lem: lemma 2.12 in OSEM}
\cite[Lemma 2.12]{burke-lewis-overton:01a}.
Assume the hypotheses of Theorem \ref{thm: Arnold Form}.
Then the gradients (see \eqref{nabla notation}) of the functions $\lam_{js}: \Cnn \to \C$
are given by
	\( \nabla \lam_{js} (\wt X) = (n_j - s)^{-1}\wt{P}^{*} J_{js}^*\wt{P}^{-*},\)
	and 
	\[	( \lam_{js}' (\wt X))^\adj \zeta   =  \zeta(n_j - s)^{-1}\wt{P}^{*} J_{js}^*\wt{P}^{-*}  
\quad\forall\ \zeta \in \C\]
with respect to the inner product $\cip{\cdot}{\cdot}_{\Cnn}$
	 (see \eqref{adjoint of derivative}).
\end{lemma}


\subsection{The derivative of the nonderogatory characteristic factors}
\label{subsec:char poly}
The local factorization of $X$ on $\Omega$ described in Theorem \ref{thm: Arnold Form} can be used to describe
a local factorization of the characteristic polynomial 
	$\det{(\lam I-X)}$ near $\wt{X}$. 
The following technical lemma allows us to represent the coefficients of the factors of the characteristic polynomial
in terms of the functions $\lam_{js}$ in Theorem \ref{thm: Arnold Form}.

\begin{lemma}
\label{lem:det lem}
Let 
	$J\in\Cnn$ 
be a Jordan matrix having ones on the superdiagonal and zeros elsewhere, 
	$\xi\in\C$, 
and
	$\lam=(\lam_0, \dots, \lam_{n-1})\in\C^n$.
Consider the matrix
	\begin{equation*}
	\xi I_n-J-
	\mbox{$\sum_{s=0}^{n-1}$} \lam_s (J^*)^s \  = \ 
	\xi I_n - 
	{\mbox{\scriptsize\(
	\begin{bmatrix}
	\lamo		&	1			&	0			&	\cdots	&	0\\
	\lam_1		&	\lam_0		&	1			&	\cdots	&	0\\
	\vdots		&	\vdots		&	\vdots		&	\ddots	&	1\\
	\lam_{n-1}	&	\lam_{n-2}	&	\lam_{n-3}	&	\cdots	&	\lamo
	\end{bmatrix}  \)}}.
	\end{equation*}
Then 
	\begin{equation}
	\label{eqn: det formula}
	\det (\xi I_n \!-\! J \!-\! \mbox{$\sum_{s=0}^{n-1}$} \lam_s(J^*)^s )
	 = 
		\xi^n \!-\! \mbox{$\sum_{s=0}^{n-1}$} (n-s)\lam_{s} \xi^{n-s-1} 
		\!+\!  
		o(\lam) 
	\end{equation}
where $J^0:=I_n$. 
\end{lemma}
\begin{proof}
First consider matrices of the form
	\[
	A_s := 
	{\mbox{\scriptsize\(
	\begin{bmatrix}
	a_0		&	-1		&	0		 		&	\cdots	&	0\\
	a_1		&	a_0		&	-1		 		&	\cdots	&	0\\
	a_2		&	a_1		&	a_0		 		&	\cdots	&	0\\
	\vdots	&	\vdots	&	\vdots 			&	\ddots	&	-1\\
	a_{s-1}	&	a_{s-2}	&	a_{s-3} 			&	\cdots	&	a_0
		\end{bmatrix}\)}
	\in \C^{s\times s}}, 
		\]
for	$s=1, \dots, n$, 
where 
	$a_0, a_1, \dots, a_{s-1} \in \C$   
so that for 
	$s=1, \dots, n-1$, $A_s$ 
is the lower right (or upper left)
	$s\times s$ 
block of 
	$A_{s+1}$.
Note that $\det(A_n)$ equals
	\begin{equation}
	\label{eqn: nice det}
	a_{n-1}\! +\! a_{n-2} \det(A_1) \!+\! a_{n-3}\det(A_2) \!+\! \cdots \!+\!  a_1 \det(A_{n-2}) \!+\! a_0\det(A_{n-1}).
	\end{equation}
To see this, expand 
the determinant on the first column of $A_n$ and observe that when the 
$s$th row and first column is deleted from $A_n$, 
the resulting 
	$(n-1)\times(n-1)$ 
matrix is lower block triangular with only two diagonal blocks, 
where the upper diagonal block is an 
	$(s-1)\times (s-1)$ 
lower triangular matrix with $-1$ in every diagonal entry and 
the lower diagonal $(n-s)\times (n-s)$ block equals $A_{n-s}$.

To establish the lemma, we need to find $\det(A_n)$ with $a_0=\xi-\lamo$ and 
	$a_s=-\lam_s$ for $s=1,\dots,n-1$. 
For 
	$n=1$, $A_1=(\xi - \lamo)$.
For $n=2$,   
	\[ \det
	\left[
	\begin{array}{rr}
	(\xi-\lamo)	& -1		\\
	-\lam_1					& (\xi-\lamo)  
	\end{array}
	\right]
	= (\xi-\lamo)^2 - \lam_1
	= \xi^2 - 2\lamo\xi - \lam_1  + o(\lam).  
	\]
Suppose \eqref{eqn: det formula} holds for 
	$s=1, \dots, n-1$. 
This together with \eqref{eqn: nice det} implies 
%
	\begin{align*}
	\det(A_n)   
	= &  
	-\lam_{n-1}- \lam_{n-2}(\xi - \lamo) 
	- \lam_{n-3}( \xi^2 - 2\lamo\xi - \lam_1  ) \ - \  \cdots 	 \\
						& \ 	\mbox{	\( - \lam_{1} (  \xi^{n-2}  
												- \sum_{s=0}^{(n-2)-1}((n-2)-s)\lam_s\xi^{(n-2)-s-1}        ) \)} 		 \\
						& \  \mbox{\(		+(\xi- \lamo) ( \xi^{n-1}  
												- \sum_{s=0}^{(n-1)-1}((n-1)-s)\lam_s\xi^{(n-1)-s-1}        )	  + o(\lam) \)}.
	\end{align*} 
Collecting like powers of $\xi$ establishes the result: 
	\begin{align*}
	\det(A_n) 
	& = \ \mbox{\small \(	 \xi^n - \lamo \xi^{n-1}  - 
	\sum_{s=0}^{(n-1)-1}((n-1)-s)\lam_s\xi^{n-s-1} -  \sum_{s=1}^{n-1}\lam_{s}\xi^{n-s-1}   +o(\lam)    \)} \\
	& =  \ \mbox{\small \(	  \xi^n - \lamo \xi^{n-1}  - (n-1)\lam_0\xi^{n-1} - 
	\left(  \sum_{s=1}^{(n-1)-1}(n-s)\lam_s\xi^{n-s-1} \right) - \lam_{n-1}  + o(\lam),	 \)}\\
	& =  \ \mbox{\small \(	 \xi^n - n\lamo \xi^{n-1}  - 
	\left(  \sum_{s=1}^{(n-1)-1}(n-s)\lam_s\xi^{n-s-1} \right) - \lam_{n-1} 	+ o(\lam),	 \)}\\
	& =  \ \mbox{\small \(	  \xi^n  -   \sum_{s=0}^{n-1}(n-s)\lam_s\xi^{n-s-1} + o(\lam). \)} 
	\end{align*}
\end{proof}

\begin{lemma}
\label{lem: little g defn and der adjoint}
Assume the hypotheses of Theorem \ref{thm: Arnold Form}.
Let $j\in\{1, \dots, m\}$ and define $g_j: \Omega \to \Pcal{n_j-1}$ by  
	\begin{eqnarray}
	\label{eqn: g defn}
	g_j( X ) :=  \Phi_{n_j}( \wh{J}_j(X) ) \ - \ \elem{n_j}{\tlam_j}\, ,
	\end{eqnarray}
where 
	$\Phi_{n_j}$ and $\elem{n_j}{\tlam_j}$ are defined in \eqref{eqn: Phi defn} and \eqref{eqn: elementary polynomial}, respectively,
and
	\begin{equation}
	\label{hat Jj defn}
	 \wh{J}_j(X):=  \tlam_j I_{n_j} + J_{j}  + 
		\mbox{$\sum_{s=0}^{n_j-1}$} \lam_{js}(X) (J_{j}^s)^*. 
	\end{equation}
Then 
	\begin{equation}
	\label{eq:expand gj}
	g_j(X) \!= 
	  -  \mbox{\small$
	 \sum_{s=0}^{n_j-1} (n_j \!-\! s)\lam_{js}(X) 
	 \elem{n_j-s-1}{\tlam_j}
			+ o(\lam_{j0}(X), \dots, \lam_{j(n_j-1)}(X))$}.
	\end{equation}
Moreover, 
	$g_j(\wt{X})=0$, 
and, with respect to the inner products  
	\[
	\cip{\cdot}{\cdot}_{\scriptscriptstyle{[n_j-1,\tlam_j]}}:=
	\mbox{   ${\sum_{s=0}^{n_j-1}}$}  \cip{\tau_{(n_j-s-1,\tlam_j)}(\cdot)}{\tau_{(n_j-s-1,\tlam_j)}(\cdot)}_{\C} 
	\]
on 
	$\Pcal{n_j-1}$ 
and 
	$\cip{\cdot}{\cdot}_{\Cnn}$ on $\Cnn$, 
	\begin{eqnarray}
	\label{eqn: nabla g adjoint defn}
		(g_j'( \wt{X} ))^\adj  
		=-  \mbox{$\sum_{s=0}^{n_j-1}$} \wt{P}^{*} J_{js}^{*} \wt{P}^{-*} \, \tau_{(n_j-s-1,\tlam_j)}, 
	\end{eqnarray}
where the mapping 
	$\tau_{(k,\lamo)}$ 
is defined in \eqref{eq:comp Taylor}. 
In particular, 
	$(g_j'(\wt{X}))^\adj$ 
is injective.
\end{lemma}
\begin{proof}
Let 
	$j\in\{1,\dots,m\}$ 
and $g_j$ be as given in \eqref{eqn: g defn}. 
Combining Theorem~\ref{thm: Arnold Form} with
Lemma~\ref{lem:det lem} where 
	$\xi:=(\lam-\tlam_j)$, 
	$J:=J_{j}$ and $n:=n_j$ 
gives \eqref{eq:expand gj}.
Since 
	$\lam_{js}(\wt X) = 0$ 
for 
	$s=0,1, \dots, n_j-1$, 
we have 
	$g_j(\wt{X}) = 0$ 
and 
	\[
	 g_j'( X )  = \mbox{$ - \sum_{s=0}^{n_j-1}$} (n_j-s)  \elem{n_j-s-1}{\tlam_j}  \lam_{j s}'(X).
	\]
Let 
	$M\in \Cnn$ 
and 
	set
	$h :=\sum_{k=0}^{n_j-1} c_k 
	\elem{n_j-k-1}{\tlam_j}
	\in \Pcal{n_j-1}$. 
Then
\begin{align*}
	&\cip{(g_j'(\wt X))^\adj h}{M}_{\Cnn} \\
	& \ \ \	= \cip{h}{g_j'(\wt X)M}_{[n_j-1,\tlam_j]} \\
	& \ \ \	= -\cip{\mbox{$\sum_{k=0}^{n_j-1}$} c_k 
	\elem{n_j-k-1}{\tlam_j}
	}
	{ \left( \mbox{$\sum_{ s=0}^{n_j-1}$}(n_j-s) 
		\elem{n_j-s-1}{\tlam_j}
		\lam'_{j s}(\wt{X})
		\right) 
		M}_{[n_j-1,\tlam_j]} \\
	& \ \ \  =- \mbox{$\sum_{s=0}^{n_j-1}$}\cip{ c_s }{  (n_j-s) \lam'_{j s}(\wt{X})M}_{\C}  
	\\ & \ \ \	    = - \mbox{$\sum_{s=0}^{n_j-1}$} \cip{ (\lam'_{j s}(\wt{X}))^\adj  c_s }{  (n_j-s) M}_{\Cnn} \\
	& \ \ \  = - \mbox{$\sum_{s=0}^{n_j-1}$}\cip{(n_j-s)^{-1}\wt{P}^{*} J_{j s}^* \wt{P}^{-*} c_s }{  (n_j-s) M}_{\Cnn} 
		\qquad\mbox{(by Lemma~\ref{lem: lemma 2.12 in OSEM})} \\
%
	& \ \ \  =- \mbox{$\sum_{s=0}^{n_j-1}$}\cip{ \wt{P}^{*} J_{j s}^* \wt{P}^{-*} c_s }{ M}_{\Cnn}   
	\\ &\ \ \ =\cip{\left(- \mbox{$\sum_{s=0}^{n_j-1}$} \wt{P}^{*} J_{j s}^* \wt{P}^{-*} 
		\tau_{(n_j-s-1,\tlam_j)}\right)h }{ M}_{\Cnn},
	\end{align*} 
%
which proves \eqref{eqn: nabla g adjoint defn}.
Moreover, the  matrices 
	\(  \{\wt{P}^{*} J_{js}^{*} \wt{P}^{-*} \}_{s=0}^{n_j-1}  \) 
are linearly independent on $\Cnn$ (over $\C$ or $\R$) and 
$\tau_{(n_j-s-1,\tlam_j)}(h)=0$ for all $s=0, \dots, n_j-1$ 
only if $h=0$. 
Hence 
	$( g_j'(\wt X))^\adj h  =0$
only if 
	$h=0$,
implying  
	$( g_j'(\wt X))^\adj$ 
is injective.
\end{proof}

We now combine the functions $g_j$ into a single function 
	$\map{\Gee}{\C\times\Omega}{\Scal_{\tilp}}$, 
where 
	${\tilp}\in\Pcal{\tiln}$ 
is the polynomial in \eqref{p-factor} with 
	$\tXi:=\{\tlam_1,\dots,\tlam_m\}$ 
as in \eqref{eqn: distinct evals matrix}:
	\begin{equation}
	\label{eqn: big G defn}
	 \Gee (\zeta,X):=(\zeta,g_1(X),\dots,g_m(X)).
	\end{equation}
Note that 
	the mapping $\Gee$ depends on $\wt X$,
	is $\C$-linear (the identity on $\C$) in its first component,
	and satisfies
	\begin{equation}
	\label{G at X tilde}
		\Gee(0,\tX) = (0,0, \dots, 0) = 0.
	\end{equation}

\begin{theorem} [Nonderogatory Characteristic Factor Derivatives]
\label{thm: G is injective}
Let the hypotheses of Theorem \ref{thm: Arnold Form} hold.
Let  
	$ \Gee : \{\zeta:|\zeta|<1\}\times\Omega \to \Scal_{\tilp}$ 
be as  in \eqref{eqn: big G defn}. 
\vn
Let 
	$R: \C^{\tiln+1} \to \C \times \Cnn$ 
be the $\C$-linear transformation given by
	\begin{equation}
	\label{R defn}
	R(v) := \left(v_\ssub{0}, \  -
	\mbox{$\sum_{j=1}^m 
	 \sum_{s=0}^{n_j-1}$} v_{\ssub{js}} \wt{P}^{*} J_{js}^{*} \wt{P}^{-*} \right), 
	\end{equation}
for all 
	$v:=(v_\ssub{0}, v_{\ssub{10}}, \dots, v_\ssub{1{(n_1-1)}}, \dots, v_\ssub{m0}, \dots, v_\ssub{m{(n_m-1)}}) \in \C^{\tiln+1}$.
Then
	\begin{equation}
	\label{eqn: nabla big G}
	( \Gee' (\zeta,\wt X))^\adj =   
	R \circ\Tau_{\tilp}
	\end{equation}
with respect to the inner products $\cip{\cdot}{\cdot}_{\Scal_\tilp}$  (see \eqref{eq:ip on Sp})
and  \(
		\cip{(\zeta,X)}{(\omega,Y)}:=\cip{\zeta}{\omega}_{\C}+\cip{X}{Y}_{\Cnn}
	\)
	on
	\(
		\C\times\Cnn \),
where 
	$\Tau_{\tilp}$ 
is defined in
\eqref{eqn: Tau defn}. 
In addition,  
	\( ( \Gee' (\zeta,\wt X))^\adj \) is injective.
\end{theorem}
\begin{proof}
The representation \eqref{eqn: nabla big G} follows immediately from 
Lemma \ref{lem: little g defn and der adjoint}. Injectivity follows from the linear
independence of the matrices 
	$J_{j s}\in\Cnn$ for $s=0,\dots,n_j-1$ 
and 
	$j=1,\dots,m$,
and the fact that
$\Tau_{\tilp}$ is a $\C$-linear isomorphism.  
\end{proof}


\section{Matrices: chain rule for the nonderogatory case}
\label{sec: nonderogatory case}
Let us now suppose that the eigenvalues in $\tXi$ (see \eqref{eqn: distinct evals matrix}) are the active 
roots (see \eqref{active roots}) of the characteristic polynomial $\Phi_n(\tX)$. 
We call
$\tXi$ the set of \emph{active eigenvalues} of $\vp$ at $\tX$, denoted by
$\Acalvp(\tX)$. 
The corresponding active indices are denoted by $\Icalvp(\tX)$.
In \eqref{p-factor}, $\tilp$ is called
the \emph{active factor} of the characteristic polynomial
	\begin{equation}
	\label{active factor defn}
	 \Phi_n(\tX) = 
	\Phi_{\tiln}(\wh{J}(\tX))  \Phi_{n_0}(B(\tX)) = \tilp\, \Phi_{n_0}(B(\tX)) , 
	\end{equation}
where 
	\[
	 \wh{J}(X):=\Diag(\wh{J}_1(X), \dots, \wh{J}_m(X)),
	\]
with $\wh{J}_j$ defined in \eqref{hat Jj defn}.
If all active eigenvalues are nonderogatory, we have
	\begin{equation*}
	\vp(X)
	=(\sff\circ F_\tilp\circ \Gee)(\zeta,X) 
	\end{equation*}
for all $X$ near $\tX$ and $\abs{\zeta}<1$. 
By \eqref{G at X tilde}, $F_\tilp ( G(0,\tX))=F_\tilp(0)=\tilp$.    
The regular and general subdifferentials of $\vp$
near $\tX$ can be obtained by computing the corresponding objects for the mapping
	$\sff\circ F_\tilp\circ \Gee$. 
The diagram in \eqref{diagram} illustrates the relationship between 
the mappings 
$\sff$, $F_{\tilp}$, $G$, and $\vp$, 
defined in
	\eqref{defn:prmf},
	\eqref{eqn: F sub p defn},
	\eqref{eqn: big G defn}, and
	\eqref{eqn: vp defn - original},
respectively.
We apply a nonsmooth chain rule to 
the representation $\sff\circ F_\tilp\circ \Gee$ to obtain
a formula for the subdifferential of $\vp$. 
\begin{theorem}[Nonsmooth Chain Rule]
\label{thm: chain rule}
\cite[Theorem 10.6]{rw:98a}
Suppose 
	$h(x) := \mathrm{g}(H(x))$
for a proper, lsc function 
	$\mathrm{g}: \R^m \to \eR$ 
and a smooth mapping 
	$H: \R^n \to \R^m$. 
Then, at any point $\tilde x \in \dom(h) = H^{-1}(\dom (\mathrm{g}))$, 
	\[ \rsd h(\tilde x) \supset  H'(\tilde x)^\adj \rsd \mathrm{g}(H(\tilde x)). \]
If 
	\begin{equation*}
	\label{eq:chain cq}
	\left( \ 
	y \in \sd^\infty \mathrm{g}(H(\tilde x))
	\mbox{ and } 
	( H'(\tilde x))^\adj y = 0
	\ \right )
	\Longrightarrow y=0,
	\end{equation*} 
then 
	\[\sd h (\tilde x) \subset ( H'(\tilde x))^\adj \sd \mathrm{g}(H(\tilde x))  
	\quad\mbox{and}\quad
	 \sd^\infty h(\tilde x)\subset ( H'(\tilde x))^\adj \sd^\infty \mathrm{g}(H(\tilde x)). \]
If, in addition, $\mathrm{g}$ is regular at $H(\tilde x)$, then $h$ is regular at $\tilde x$, 
	\[ \sd h(\tilde x) = ( H'(\tilde x))^\adj \sd \mathrm{g}(H(\tilde x))  
	\quad \mbox{and} \quad
	 \sd^\infty h(\tilde x) = ( H'(\tilde x))^\adj \sd^\infty \mathrm{g}(H(\tilde x)). \]
\end{theorem}

\begin{theorem}[Chain Rule for $\vp$]
\label{thm:vp is subdifferentially regular}
Let $f:\C\to \eR$ be proper, convex, lsc and 
let $\vp$ be as in \eqref{eqn: vp defn - original}. 
Suppose that
	$\wt{X}\in\Cnn$ 
is such that the notation of Section~\ref{sec: derivative of characteristic factor} holds,
$f$ satisfies either \eqref{f-twice-diffable-and-psd} or \eqref{rspan is C}
at $\tlam_j$ for each $j\in \{1, \dots, m\}$,
and 
	$\Phi_n(\wt X)$ 
has active factor 
	\(
		\til p 
	\)
given by \eqref{p-factor} and satisfying \eqref{eq:not zero}. 
If $\tlam_j$ is nonderogatory for all
$j \in \{1,\dots,m\}$, then
	\[
		\rsd \vp(\wt{X}) \supset  \set{Y\in\Cnn}{ (0,Y)\in R(D_{\tilp})},
	\]
where 
	$D_{\til p}$ and $R$ are 
as defined in Theorem \ref{Thm: Burke-Eaton 12} and \eqref{R defn}, respectively.
If $f$ satisfies \eqref{f-twice-diffable-and-psd} at $\tlam_j$ for all $j\in\{1, \dots, m\}$, then $\vp$ is subdifferentially regular at 
$\wt X$ with
	\begin{equation*}
	\begin{aligned}
	\sd \vp(\wt{X})&= 
	\set{Y\in\Cnn}{ (0,Y)\in R(D_{\tilp})},
	\quad\mbox{and} \\  
	\sd^\infty \vp(\wt{X}) &=
	\set{Y\in\Cnn}{ (0,Y)\in R(D_{\tilp}^\infty)},  
	 \end{aligned}
	 \end{equation*}
where $D_{\til p}^\infty$ is as defined in Theorem \ref{Thm: Burke-Eaton 12}.
\end{theorem}
\begin{proof}
We prove only the second statement 
since the proof of the first statement is similar.
Define  $\map{H}{\{\zeta:\abs{\zeta}<1\} \times\Omega}{\Pcal{\tiln}}$ by $H:=F_{\tilp} \circ \Gee $ so that, 
in particular, 
	\( 
	H(0,\tX) = F_{\tilp}(G(0,\tX)) = F_{\tilp}(0)=\tilp .
	\)
By Theorem~\ref{thm: G is injective}, the adjoint of the derivative of $H$
at $(0,\tX)$ is given by     
	\begin{equation*}
	\begin{aligned}
		( H'(\zeta,\tX))^{\adj} 
		& = (G'(\zeta,\tX))^\adj \circ (F_{\tilp}'(G(\zeta,\tX)))^\adj
		\\ &=R \circ \Tau_{\tilp}  \circ (F_{\tilp}'(G(\zeta,\tX)))^\adj   && \mbox{by \eqref{eqn: nabla big G}}.
	\end{aligned}
	\end{equation*}
Therefore,
	\begin{equation}\label{H adjoint}
	\begin{aligned}
		( H'(0,\tX))^{\adj} 
		& = R \circ \Tau_{\tilp}\circ (F_{\tilp}'(0))^{\star} && \mbox{by \eqref{G at X tilde}}
		\\ & = R \circ \Tau_{\tilp} \circ (F_{\tilp}'(0))^{-1}&& \mbox{by \eqref{F' inverse is F' adjoint}}.
	\end{aligned}
	\end{equation}
Since $\Tau_{\tilp}$, $R$ and $F'_\tilp(0)$ are injective, so is $(H'(0,\wt X))^{\adj}$.

Define $\map{\hat\vp}{\{\zeta:\abs{\zeta}<1\} \times\Omega}{\eR}$ by 
	\[
		\hat\vp(\zeta,X) := (\sff \circ H)(\zeta,X) \quad \forall \ (\zeta, X) \in \{\zeta:\abs{\zeta}<1\} \times\Omega,
	\]
so that 
	\begin{equation}
	\label{H property} 
	\vp(X)\! =\! \hat\vp(0,X)\!=\! \hat\vp(\zeta,X)\!=\! (\sff \circ H)(\zeta,X)
	\quad\! \forall \
	(\zeta, X) \in \{\zeta:\abs{\zeta}<1\} \times\Omega.
	\end{equation}
Take $\mathrm{g} := \sff$ and $h := \hat\vp$ in Theorem~\ref{thm: chain rule}.
Then by Theorems~\ref{Thm: Burke-Eaton 12} and~\ref{thm: chain rule}, $\hat\vp$ is subdifferentially regular at $(\zeta,\tX)$ and 
	\begin{equation*}
	\begin{aligned}
	  \sd \hat\vp(\zeta,\tX) 
						   &=\sd \hat\vp(0,\tX) 	&&\mbox{by \eqref{H property}}\\
	  					   & = ( H'(0, \tX) )^\adj \sd \sff ( H(0, \tX ) )   \\
						   & = R \circ \Tau_{\tilp} \circ (F_{\tilp}'(0))^{-1}  \sd \sff ( H(0, \tX ) )   
						   &&\mbox{by \eqref{H adjoint}}\\
						   & = R \circ \Tau_{\tilp} \circ (F_{\tilp}'(0))^{-1}  F'_{\tilp}(0)\circ\Tau_{\tilp}^{-1}  (D_{\tilp}) 
						   &&\mbox{by \eqref{polynomial subdifferential}}\\
						   & = R(D_{\tilp}). 
		\end{aligned}
	\end{equation*}
Similarly,
	  \( \sd^\infty\hat\vp(\zeta,\tX)  = R(  D_{\tilp}^\infty).  \)
By \eqref{H property}, $\hat\vp(\zeta,\tX)$ is constant on $\{\zeta:\abs{\zeta}<1\}$ so
\cite[Corollary 10.11]{rw:98a} on partial subdifferentiation tells us that
the first component of every element of $\sd \hat\vp(0,\tX)$ is zero.
Hence, the final result also follows from \cite[Corollary 10.11]{rw:98a}.
\end{proof}
\begin{remark}
This result recovers the subdifferential regularity of the spectral abscissa \cite[Theorem 8.1]{burke-overton:01b} and
the formula for its subdifferential in \cite[Theorem 7.2]{burke-overton:01b}
in the case of nonderogatory active eigenvalues.
\end{remark}

\begin{corollary}[Explicit Subdifferential Representation]
\label{cor:nonderogatory explicit}
Assume the hypotheses of Theorem~\ref{thm:vp is subdifferentially regular}  
with \eqref{f-twice-diffable-and-psd} holding.
Then
	$Y \in \sd\vp(\wt X)$
if and only if there exists
 	$U_j \in \C^{n_j\times n_j}$ for $j\in\{1, \dots, m\}$ 
such that 
	\[ \wt P^{-*} Y \wt{P} = \Diag(0_{n_0\times n_0}, U_1 \dots, U_m), \]  
where 
$U_j$ is lower triangular Toeplitz with diagonal entries 
	$-\mu_{j1}$, subdiagonal entries $-\mu_{j2}$, 
and parameters $\gam_j\ge 0$
satisfying	$\sum_{j=1}^m \gam_j=1$, 
\[	
\mu_{j1} \!=\!\gam_j  (\nabla f(\tlam_j)/ n_j)\ \text{and}\ 
\langle {  - \mu_{j2} },{ \nabla f( \tlam_j)^2} \rangle_\C\! \leq \!
	 (\gam_j / n_j ) f''(\tlam_j;  \sfi \nabla f ( \tlam_j),  \sfi \nabla f (\tlam_j)).
\]
\end{corollary}

\section{Matrices: the general case}
\label{sec: general case}
In Section~\ref{sec: nonderogatory case}, we use the Arnold form (see Theorem~\ref{thm: Arnold Form}) to simultaneously
derive a representation for the subdifferential and establish the subdifferential regularity 
of convexly generated spectral max functions at matrices with nonderogatory active eigenvalues. 
In this section, we extend techniques developed in \cite[Theorem 7.2]{burke-overton:01b}
for the spectral abscissa to derive a representation for the regular subdifferential of spectral max functions
even when some active eigenvalues are derogatory. 
We do not derive formulas for the general (limiting) subdifferential (see \eqref{general-sd-formula}). 
Formulas in the derogatory case are, for the most part, 
unknown even for the spectral abscissa\footnote{Grundel and Overton derive the general subdifferential in the simplest derogatory, defective 
case when $n=3$ for the spectral abscissa in \cite{grundel:13}. The nondefective (or semisimple) case is treated in \cite[Theorem 8.3]{burke-overton:01b}.}.

\subsection{Review of results from \cite{burke-overton:01b}}
\label{ss: review}
Define 
	\(\Lam: \Cnn \to \C^n\) 
to be the mapping that takes a matrix $X$ to its vector of eigenvalues, repeated according to algebraic multiplicity and ordered lexicographically (see \S\ref{factorization space section}).  
Any function of the form
	\( \psi:= \mathrm{h} \circ \Lam \),
where
	\( \mathrm{h}: \C^n \to \eR \) is invariant under permutations of its arguments,
is a spectral function (see \S\ref{intro}).  

Now assume that the set
	$\tXi  = \{  \tlam_1, \dots,  \tlam_m\}$ in \eqref{eqn: distinct evals matrix}
is the complete set of distinct eigenvalues of  
	$\wt X \in \Cnn$
	with algebraic multiplicities $n_1, \dots, n_m$ (so that $n=\sum_{j=1}^m n_j$) and geometric multiplicities $q_1, \dots, q_m$.
We use the Jordan form notation described in \eqref{eqn: P and J, 1} and \eqref{Nilpotent block jk},
where the matrix $\wt B$ is no longer present.
Recall the following results from \cite{burke-overton:01b}.

\begin{theorem}
\label{theorems A B}
\cite[Theorems 4.1 and 4.2]{burke-overton:01b} 
Let 
	$\wt X$ 
be as given above and let $\psi$ be a spectral function.
If $Y$ is a subgradient or horizon subgradient of $\psi$ at $\wt X$, then 
	\begin{equation}
	\label{W defn}
		W := \wt P^{-*} Y \wt P^*
	\end{equation}
satisfies
	\begin{equation}
	\label{W block diag}
		W = \Diag(W_1, \dots, W_m), 
	\end{equation}
where	
	\begin{equation}
		\label{eqn: Wj}
		W_j :=
 		\left[ 
		\begin{array}{c  c c}
		 W^{(11)}_{j}	& \cdots & W^{ ( 1 q_j ) }_{ j }  \\ 
			\vdots	& \ddots  &\vdots \\ 
		 W^{(q_j  1)}_{j}&  \cdots & W^{ ( q_j q_j ) }_{j}  \\ 
		\end{array} \right],
	\end{equation}
with
	\( W^{(rs)}_{j} \) 
a rectangular 
	\( m_{jr}\times m_{js} \) 
lower triangular Toeplitz matrix for 
	\( r=1, \dots, q_j , \)
	\( s=1, \dots, q_j , \) 
	\( j=1, \dots, m . \) 

If $Y$ is further assumed to be a regular subgradient of $\psi$  at $\wt X$, 
then 
	\begin{equation}
	\label{eqn: Y and W, 2}
		W_j = \Diag(W^{(11)}_j, \dots, W^{( q_j q_j ) }_j ),
	\end{equation}
where 
	\( W^{(kk)}_{j} \) 
is an 
	\( m_{j k}  \times m_{jk}\)
lower triangluar Toeplitz matrix with diagonal  
	\( \theta_{j1}   \) 
and subdiagonals  
	\( \theta_{js} , \) 
$s=2,\dots,m_{jk},\ k=1,\dots,q_j,\ j=1,\dots,m$,
and $\sum_{k=1}^{q_j} m_{jk}=n_j$.

\end{theorem}
\begin{remark}
By lower triangular Toeplitz,  
we mean that the value of the $k, \ell$ 
entry of $W^{(rs)}_j$
depends only on the difference $k - \ell$ (is constant along the diagonals), 
and is zero if $k < \ell$  or $m_{jr} - k> m_{js} - \ell$ 
(is zero above the main diagonal, drawn from either the top left or bottom right of the block).
\end{remark}

Observe that  
	$\Phi_n$  (see  \eqref{eqn: Phi defn}) 
is smooth since each coefficient is a polynomial in the entries of $X$. 
The next result provides an expression for one-sided $\C$-derivatives of $\Phi_n$ with respect to a given direction.

\begin{lemma} 
\label{lem: 7.1}
\cite[Lemma 7.1]{burke-overton:01b}
Let
	\( \wt X \in \Cnn \)
and 
	\( Z \in \Cnn \) 
be given, and assume $\wt X$ has Jordan form \eqref{eqn: P and J, 1}. 
Define
	$p \in \Mcal{n}$ 
by 
	\begin{equation*}
	\label{eqn: char poly}
		p := \Phi_n(\wt X) =  
		\mbox{$\prod_{j=1}^m$} \elem{n_j}{\tlam_j}, 
	\end{equation*}
where $\Phi_n$ and $\elem{\ell}{\lamo}$ are defined in \eqref{eqn: Phi defn}
and \eqref{eqn: elementary polynomial}, respectively.
Set
	\begin{equation}
	\label{eqn: V defn for q}
		V:= \wt P Z \wt P^{-1},
	\end{equation}
and let $V_{jj}$ be the $n_j \times n_j$ diagonal block of $V$ 
corresponding to block $J_j$ of the matrix $J$
defined in \eqref{eqn: P and J, 1}. 
Then the action of the $\C$-derivative $\Phi_n'(\tX)$ is  
	\begin{equation}
	\label{Phi' action}
	\Phi_n'(\wt X)Z 
	=
		- \mbox{$\sum_{j=1}^m$}
		 \left(  
		 	\mbox{$\prod_{\mbox{\tiny\(\begin{array}{c}k\!=\!1,  k\!\neq\! j \end{array}\)} }^{m}$} 
			\elem{n_k}{\tlam_k}  
		\right) 
		\left( 
		\mbox{$\sum_{\ell=1}^{m_j}$}
		\tr \left(
			N_{[j]}^{ \ell -1} V_{jj} \right) 
			\elem{n_j-\ell}{\tlam_j}
		\right)
	\end{equation}
for all $Z \in \Cnn$, where 
	\begin{equation}
	\label{N sub bracket} 
	N_{[j]} := \Diag(N_{j1}, \dots, N_{jq_j})
	\end{equation}
with $N_{jk}$ defined in \eqref{Nilpotent block jk}.
$N_{[j]} := \Diag(N_{j1}, \dots, N_{jq_j})$ with $N_{jk}$ defined in \eqref{Nilpotent block jk}.
\end{lemma}


\begin{theorem}
\label{thm: section 6}
Suppose $\vp$ is as in \eqref{eqn: vp defn - original}, 
where 
	\begin{equation}
	\label{eq:non-zero continuous grad}
	\nabla f \mbox{ is continuous at $\tlam_j$ with }\nabla f(\tlam_j)\ne 0\ \forall\ j\in \Icalvp(\wt X).
	\end{equation}
Define  
	\begin{equation}
	\label{eqn: sigj defn - max}
		\sig_j :=\begin{cases}   \theta_{j1} / \nabla f ( \tlam_j) ,  & j\in\Icalvp(\wt X) \\
		0, & j\notin\Icalvp(\wt X),
		\end{cases}	
	\end{equation}
where $\theta_{j1}$ is given in  \eqref{eqn: Y and W, 2}, and set
	\(	\sig := [\sig_1, \dots, \sig_1, \dots, \sig_m, \dots, \sig_m]^T \in \C^n, \)
where each $\sig_j$ is repeated $n_j$ times. 
Let $\wt X$ have Jordan form \eqref{eqn: P and J, 1},  
set
	\begin{equation}
	\label{simplex}
	  \simplex^{n-1}  :=  
	 \{ \gam\in\R^n \ \vert \ 0 \leq \gam_i, \ i=1, \dots, n, \ \mbox{$\sum_{i=1}^n \gam_i = 1$} \} ,   
	 \end{equation}
and assume that $Y$ is a regular subgradient for $\vp$ at $\wt X$.
Then $Y$ satisfies 
\eqref{eqn: Wj} and \eqref{eqn: Y and W, 2}, 
and we have the following:
%
\begin{enumerate}[label=(\alph*)]
	\item \cite[Theorem 6.1]{burke-overton:01b} 
		For all $j\in\{1, \dots, m\}$,
		\begin{equation}
		\label{eqn: theorem 6.1 vanlsf}
			 j \notin \Icalvp(\wt X) \quad \implies \quad  W_j=0, 
		\end{equation}
	where $W_j$ is given in \eqref{eqn: Y and W, 2}.
	\item \cite[Theorem 5.2, Equations 6.5, 6.6]{burke-overton:01b}
		We have
		\(	\sig \in \simplex^{n-1} \)
		with 
		\(	\sig_j = 0 \ \mbox{if } j \notin \Icalvp(\wt X). \)
		In particular, this implies that
		\begin{gather}
		\label{eqn: sig 1}
			\mbox{\( 	\sig_j \geq 0\quad\mbox{and} \quad \sum_{j\in\Icalvp(\wt X)} n_j \sig_j =1 . 	\)}
		\end{gather}
	\item \cite[Theorem 5.3]{burke-overton:01b}
	If, in addition, the second $\R$-derivative of $f$ is continuous at each 
		   $\tlam_j$ for  $j\in \Icalvp(\wt X)$, then 
		\begin{equation}
		\label{eqn: theta f' inner product - max}
			\langle { \theta_{j2} } ,{ (\nabla f ( \tlam_j))^2} \rangle_\C \geq -\sig_j \eta_j 
			\quad\mbox{ whenever $m_j  \geq 2$}, 
		\end{equation}
		where $\sig_j$ is given in \eqref{eqn: sigj defn - max}, $\theta_{j2}$ is given in \eqref{eqn: Y and W, 2}, $m_j$ is given in \eqref{mj defn},
		and 
		\begin{equation}
		\label{etaj defn}
		\eta_j := f''(\tlam_j;  \sfi  \nabla f ( \tlam_j)  , \sfi  \nabla f ( \tlam_j) ).
		 \end{equation}   
\end{enumerate}
\end{theorem}

\subsection{The regular subdifferential and its recession cone}
\label{ss: regular subdifferential}
The main result of Section~\ref{sec: general case} now follows.
\begin{theorem}[Regular subdifferential formula]
\label{thm: new 7.2 - derogatory}
Let $\wt X$ have Jordan form \eqref{eqn: P and J, 1} with the distinct eigenvalues $\tXi$ 
having geometric multiplicities $q_1, \dots, q_m$.
Suppose $f$ satisfies 
\eqref{f-twice-diffable-and-psd}  at each $\tlam_j$  
for $j\in \Icalvp(\wt X)$ and \eqref{eq:non-zero continuous grad} holds.
Then 	
	$Y \in \rsd \vp(\wt X)$
if and only if $Y$ satisfies 
	\eqref{W defn}-\eqref{eqn: Y and W, 2}
	and
	\eqref{eqn: theorem 6.1 vanlsf}-\eqref{eqn: theta f' inner product - max}.
\end{theorem}
\begin{proof}  
We follow the pattern of proof established in  \cite[Theorem 7.2]{burke-overton:01b} for the spectral abscissa.
Let 	
	\( Y \in \rsd \vp(\wt X) \). 
By Theorems~\ref{theorems A B} and \ref{thm: section 6},  
$Y$ satisfies 
\eqref{W defn}-\eqref{eqn: Y and W, 2} and
%
\eqref{eqn: theorem 6.1 vanlsf}-\eqref{eqn: theta f' inner product - max}, respectively.

Next suppose that $Y\in \Cnn$ satisfies conditions
	\eqref{W defn}-\eqref{eqn: Y and W, 2}
and
	\eqref{eqn: theorem 6.1 vanlsf}-\eqref{eqn: theta f' inner product - max}.
We prove that $Y$ is a regular subgradient, that is,
	\begin{equation}
	\label{eqn: subdifferential inequality}
		\ip{Y}{Z}_{\Cnn} \leq \dee \vp(\wt X)(Z), \quad\mbox{for all} \quad Z\in\Cnn.
	\end{equation}
By applying the chain rule (see Theorem~\ref{thm: chain rule}), we obtain 
	\begin{equation}
	\label{temp}
		\dee \vp(\wt X)(Z) 
		= \dee (\sff \circ \Phi_n)(\wt X)(Z) 
		\geq \dee\sff(\Phi_n(\wt X)) ( \Phi_n'(\wt X)Z),	
		\end{equation}
where the action $\Phi_n'(\wt X)Z$ is given in \eqref{Phi' action}
and $\sff$ is defined in \eqref{defn:prmf}.
Set $p:=\Phi_n(\wt X)$ in \eqref{temp}  
to obtain
	\begin{equation}
	\label{eqn: old 7.14}
		\dee \vp(\wt X)(Z) \geq \dee \sff( p) (\Phi_n'(\wt X)Z).
	\end{equation}
By 
	\eqref{W defn}-\eqref{eqn: Y and W, 2}
and
\eqref{eqn: theorem 6.1 vanlsf},
	\[
	\ip{Y}{Z}_{\Cnn} = \ip{W}{V}_{\Cnn} = \mbox{$\sum_{j\in\Icalvp(\wt X)} \sum_{\ell=1}^{m_j}$}  
	\Realpart
		 (
			\overline{ \theta_{ j\ell }} \tr  ( N_{[j]}^{\ell-1} V_{jj}  )
		 )  , 
	\]
where $V$ is defined in \eqref{eqn: V defn for q}.
Set
	\( \om_{j\ell}:= -\tr (N_{[j]}^{\ell-1} V_{jj}  ) \)
for 
	 $\ell=1, \dots, m_j$, \( j\in\Icalvp(\wt X) \).
Then 
	\[
	\ip{W}{V}_{\Cnn} = \mbox{$\sum_{j\in\Icalvp(\wt X)} \sum_{\ell=1}^{m_j}$}  
	\Realpart
		 (
			\overline{ \theta_{ j\ell }}  (-\om_{j\ell}) 
		 )  .
	\]
If either \eqref{tc-second} or \eqref{tc-third} in Theorem~\ref{subderivative general} fails for some $j\in\Icalvp(\wt X)$, then 
the subderivative
	\( \dee \sff(p)(\Phi_n'(\wt X)Z)\)
is infinite, in which case \eqref{eqn: subdifferential inequality} is trivially satisfied.
If \eqref{tc-second} and \eqref{tc-third} in Theorem~\ref{subderivative general} hold for all 
	\( j\in \Icalvp(\wt X) , \) 
then
 	\begin{alignat}{2}
	\nonumber
		\ip{Y}{Z}_{\Cnn}
		& \!=\! \mbox{$\sum_{j\in\Icalvp(\wt X)}$}   ( \,
		\Realpart (\overline{\theta_{j1}}  (-\om_{j1})  )  
		\!+\! \Realpart (\overline{\theta_{j2}}(- \om_{j2}) )   \,),  
		&& \\ 
	\nonumber
		&\!=\! 
		\mbox{ \(
		\sum_{j\in\Icalvp(\wt X)}  
		\Realpart (\overline{\sig_j \nabla f ( \tlam_j)} (-\om_{j1}) )  
		\!-\! \Realpart (\overline{\theta_{j2}}  \,   \om_{j2}   )    \)}
		&& \ \mbox{by \eqref{eqn: sigj defn - max},} \\ 
	\nonumber
		&\!=\! 
		\mbox{ \(
		\sum_{j\in\Icalvp(\wt X)}  
		\sig_j\Realpart (\overline{ \nabla f ( \tlam_j)} (-\om_{j1}) )  
		\!-\!\Realpart (\overline{\theta_{j2}} \,   \om_{j2}   )    		\)}
		&& \ \mbox{by \eqref{eqn: sig 1},  } \\ 
	\label{eqn: pause for 2nd term}
		&\!=\! 
		\mbox{ \(
		\sum_{j\in\Icalvp(\wt X)}   
		\sig_j    f' ( \tlam_j; -\om_{j1})
		\!-\!\Realpart\left(\overline{\theta_{j2}} \,    \om_{j2} \right)  \)}  
		&& \ \mbox{by \eqref{subderivative is directional}}    ,
	\end{alignat}
where the second term in each line does not appear if $m_j=1$ (see \eqref{mj defn}).  
Recall that for all $j\in\Icalvp(\wt X)$, $\om_{j2}$ satisfies \eqref{tc-second}, which, 
since $f$ satisfies \eqref{f-twice-diffable-and-psd}  
at $\tlam_j$,  
is equivalent to 
  	\begin{equation}
	\label{eqn: omj2 tj}
		\om_{j2} = t_j (\nabla f ( \tlam_j))^2\quad\mbox{for some $t_j \geq 0$}
 	\end{equation} 
by \cite[Lemma 4]{burke-lewis-overton:05a}.
Therefore,
	\begin{align}
	\nonumber
		\langle{\theta_{j2}} , {f'( \tlam_j)^2} \rangle_{\C} & \geq -\sig_j \eta_j 
		&&\!\!\mbox{by \eqref{eqn: theta f' inner product - max}, which implies}   \\
	\nonumber
		\Realpart (\overline{\theta_{j2}} \, \om_{j2}  )  
		& \geq - t_j\sig_j \eta_j 
		&&\!\!\mbox{by \eqref{eqn: omj2 tj},  which implies } \\
	\nonumber
		-\Realpart (\overline{\theta_{j2}} \, \om_{j2}  ) 
		& \leq  t_j \sig_j f''( \tlam_j; \sfi \, f'( \tlam_j), \sfi \, f'( \tlam_j)), 
		&&\!\!\mbox{by \eqref{etaj defn}, which implies}\\
	\label{eqn: plug back in}
		-\Realpart (\overline{\theta_{j2}} \, \om_{j2}  )  
		& \leq  \sig_j f''( \tlam_j;    \sqrt{-\om_{j2}},   \sqrt{-\om_{j2}}\,).
	\end{align}
Plugging \eqref{eqn: plug back in} into  \eqref{eqn: pause for 2nd term} yields 
	\begin{alignat*}{2}
	\nonumber
		\ip{Y}{Z}_{\Cnn}
		&\!\!= 
		\mbox{\(
		\sum_{j\in\Icalvp(\wt X)}  
		(  
		\sig_j   f'( \tlam_j; -\om_{j1})
		-\Realpart (\overline{\theta_{j2} }    \om_{j2}   )   
		) 
		\)}
		\\
		&\!
		\mbox{\(
		\leq \sum_{j\in\Icalvp(\wt X)} (  
			\sig_j  f'( \tlam_j; -\om_{j1})
		+\sig_j f''( \tlam_j;    \sqrt{-\om_{j2}},    \sqrt{-\om_{j2}} )   
		) 
		\)}  \\
		&\! \!= 
		\mbox{\(
		\sum_{j\in\Icalvp(\wt X)}  \sig_j n_j
		(   
		f'( \tlam_j; -\om_{j1})
		+  f''( \tlam_j; \sqrt{-\om_{j2}},  \sqrt{-\om_{j2}} )
		) /  n_j  
		\)}
		  \\
		&\! \leq 
		\mbox{\(
		\max_{j\in\Icalvp(\wt X)}   
		\!\left\{( 
		f'( \tlam_j; -\om_{j1})
		+  f''( \tlam_j;   \sqrt{-\om_{j2}},   \sqrt{-\om_{j2}} ) )/n_j 
		\right \}
		\)}
		&&\mbox{ by \!\eqref{eqn: sig 1}}\\
		&\! = \dee \sff( p) (\Phi_n'(\tX)Z),
	\end{alignat*}
where  the last equality holds by Theorem~\ref{subderivative general}.
Combining this with \eqref{eqn: old 7.14} gives \eqref{eqn: subdifferential inequality}.
\end{proof}

\begin{remark}
Theorem~\ref{thm: new 7.2 - derogatory} establishes the formula for the regular subdifferential of $\vp$ 
at a matrix with any number of derogatory or nonderogatory active eigenvalues.
The theorem recovers the formula for the regular subdifferential in the case where all active eigenvalues are nonderogatory
(see Corollary~\ref{cor:nonderogatory explicit}).
In Section 7, it is shown that although
Theorem~\ref{thm: new 7.2 - derogatory} does not directly apply to the spectral radius, it can still be used
to derive its regular subdifferential.
\end{remark}

\begin{theorem}
[Recession cone of the regular subdifferential]
\label{thm: recession cone}  
Assume the hypotheses of Theorem \ref{thm: new 7.2 - derogatory}.  
Then $Y\in\rsd\vp(\wt X)^\infty$
if and only if $Y$ satisfies 
	\eqref{W defn}-\eqref{eqn: Y and W, 2},
	\eqref{eqn: theorem 6.1 vanlsf}, 
and, for
	$j \in \Icalvp(\wt X)$, 
the diagonal entries 
	$\theta_{j1}$ in \eqref{eqn: Y and W, 2} satisfy
	\begin{equation}
	\label{theta1s are zero}
	\theta_{j1}=0,
	\end{equation}
and, if in addition $m_j\geq 2$ (see \eqref{mj defn}), the subdiagonal entries 
	$\theta_{j2}$ in \eqref{eqn: Y and W, 2} satisfy
	\begin{equation}
	\label{eqn: Z subdiagonal}
	 \langle 	{\theta_{j2}}, { (\nabla f ( \tlam_j))^2} \rangle_\C \geq 0. 
	\end{equation}
\end{theorem}
\begin{proof}
First suppose 
	\( Y \in \rsd \vp(\wt X)^\infty . \)
Since $\rsd \vp(\wt X)$ is convex,  
	\( Z + tY \in  \rsd \vp(\wt X) \)
for all 
	\( Z  \in \rsd \vp(\wt X)$ and $t \geq 0   \) 
by the definition of recession cone.
By Theorem~\ref{theorems A B}, both 
	$\wt P^{-*} Z \wt P^*$ 
and 
	$\wt P^{-*} (Z+tY) \wt P^*$ 
satisfy 
	\eqref{W defn}-\eqref{eqn: Y and W, 2} 
for all 
	$Z\in\rsd\vp(\wt X)$ 
and 
	$t\geq0$.
In particular, both 
	$\wt{P}^{-*} Z \wt{P}^*$ 
and 
	$\wt{P}^{-*} (Z+tY) \wt{P}^*$ 
are block diagonal matrices, with each block a lower triangular Toeplitz matrix 
where the blocks corresponding to the same eigenvalue have the same diagonal entries.
Therefore,
	$\wt P^{-*} Y \wt P^*$ 
has the block structure specified in \eqref{W defn}-\eqref{eqn: Y and W, 2}.
By Theorem~\ref{thm: section 6}, $Y$ also satisfies \eqref{eqn: theorem 6.1 vanlsf}.

For 
	$Z \in \rsd \vp (\wt X)$, 
denote the diagonal entries of the $j$th block of 
	$\wt P^{-*} Z \wt P^*$ 
by 
	$z_{j1}$. 
By  Theorem~\ref{thm: section 6}, $Z+tY$ satisfies  
\eqref{eqn: sig 1} for all $t\geq 0$, that is,
	\[
		\mbox{$\sum_{j\in \Icalvp(\wt X)}$}  n_j    (z_{j1} + t \theta_{j1}) / \nabla f ( \tlam_j)  = 1
	\] 
and
	\(
	 (z_{j1} + t \theta_{j1}) / \nabla f ( \tlam_j) \geq 0 \)
	for all
	$j \in\Icalvp(\wt X)$ and
$t\geq 0$.  Therefore, 
	$\theta_{j1} = 0$ 
for all 
	$j \in\Icalvp(\wt X)$, proving \eqref{theta1s are zero}.
If, in addition, $m_j\geq 2$ (see \eqref{mj defn}), denote the subdiagonal entries of $\wt P^{-*} Z \wt P^*$  by 
	$z_{j2}$.
By Theorem~\ref{thm: section 6}, $Z+tY$ satisfies
	\eqref{eqn: theta f' inner product - max}, that is, 
	\[ 
		\langle { (z_{j2} + t \theta_{j2}) } ,{ (\nabla f ( \tlam_j))^2 } \rangle_\C \geq      -(z_{j1} + t \theta_{j1})\eta_j   / \nabla f ( \tlam_j) 
	\]
for all 
	$j \in \Icalvp(\wt X)$ 
and 
	$t \geq 0$.
Since $\theta_{j1}=0$, this becomes 
	\[
		\langle { z_{j2}  } ,{ ( \nabla f ( \tlam_j))^2 } \rangle_\C  + t \langle {    \theta_{j2} } ,{ (\nabla f ( \tlam_j))^2 } \rangle_\C 
		\geq -   z_{j1}   \eta_j/ \nabla f ( \tlam_j)    
	\]
for all $t \geq 0$ and $Z \in \rsd\vp(\wt X)$.
Since, by \eqref{eqn: theta f' inner product - max}, 	
	\[ \langle { z_{j2}  } ,{ (\nabla f ( \tlam_j))^2 } \rangle_\C  \geq -   z_{j1}   \eta_j/ \nabla f ( \tlam_j)   
	\]
	for all $j \in \Icalvp(\tX)$ with $m_j\geq 2$, 
	\eqref{eqn: Z subdiagonal} follows. 

Now suppose $Y \in \C^{n\times n}$ satisfies 
	\eqref{W defn}-\eqref{eqn: Y and W, 2},
	\eqref{eqn: theorem 6.1 vanlsf}, 
	\eqref{theta1s are zero}, and, for $j\in \Icalvp(\tX)$ with $m_j\geq 2$,
	\eqref{eqn: Z subdiagonal}.
Then for any $t \geq 0$ and $Z \in \rsd \vp(\wt X)$, the matrix $Z+tY$ satisfies the conditions of Theorem \ref{thm: new 7.2 - derogatory},
which implies that $Z+tY \in \rsd \vp(\wt X)$,  proving the reverse inclusion.
\end{proof} 

\subsection{Nonderogatory active eigenvalues are necessary for regularity}
\label{ss: nonderogatory necessary}
Theorem~\ref{thm:vp is subdifferentially regular} shows that nonderogatory active eigenvalues are 
sufficient for the subdifferential regularity of $\vp$.
We now show this condition is also necessary.

\begin{lemma}
\label{thm: new 8.2} 
In \eqref{eqn: distinct evals matrix}, let 
	$\tXi$ 
be the complete set of distinct eigenvalues 
of
	$\wt X \in \Cnn$,
and suppose $f$ satisfies \eqref{eq:non-zero continuous grad}.
 If 
 	$\tlam_j$ 
is derogatory for some 
	$j\in \Icalvp(\wt X)$, 
then 
	\( \sd\vp(\wt X) \supsetneq \rsd\vp(\wt X). \)
\end{lemma}
\begin{proof} 
We follow the method of proof in \cite[Theorem 8.2]{burke-overton:01b} for the spectral abscissa.
Let 
	$j\in \Icalvp(\wt X)$ 
be such that 
	$ \tlam_j$ 
is derogatory. 
Then the geometric multiplicity of 
	$ \tlam_j$ 
is at least 2. 
Since $f$ is differentiable at 
	$ \tlam_j$, $ \tlam_j \in \interior(\dom(f))$. 
This together with the fact that 
	$\nabla f ( \tlam_j) \neq 0$ 
implies 
	$ \tlam_j$ 
is not a local maximizer of $f$. 
Therefore, there exists 
	$\lam^\nu \to  \tlam_j$ with $
	f(\lam^\nu)>f( \tlam_j)$. 
Set 
	$\beta^\nu := \lam^\nu- \tlam_j$ 
and define
	\( X^\nu := \wt P^{-1} (J + \beta^\nu E) \wt P \)
where $\wt  P$ and $J$ are defined in \eqref{eqn: P and J, 1},
and the entries of 
$E$ are one in the $m_{j1}$ diagonal positions corresponding to the Jordan sub-block 
	$J_j^{(1)}$
with all other entries zero.
Clearly $X^\nu \to \wt X$.
The only active eigenvalue of 
	$X^\nu$ 
is 
	$\lam^\nu$, 
with multiplicity 
	$m_{j1}$,
and 
	\(J + \beta^\nu E\) 
is the Jordan form of 
	$X^\nu,$
up to re-ordering of the Jordan blocks. 
By Theorem \ref{thm: new 7.2 - derogatory}, 
	$\rsd\vp(X^\nu)$ 
includes the matrix
	\( ( \nabla f (\lam^\nu) /  m_{j1} ) \wt P^* E \wt P^{-*} \)
for all $\nu$, which implies 
	\( (\nabla f ( \tlam_j) /m_{j1}) \wt P^* E \wt P^{-*} \in \sd\vp(\wt X) . \) 
However, 
	\(  (\nabla f ( \tlam_j)/m_{j1}) \wt P^* E \wt P^{-*} \notin \rsd\vp(\wt X) \)  
since 
	\( (\nabla f ( \tlam_j) / m_{j1}) \wt P^* E \wt P^{-*} \)
does not satisfy 
the conditions in  Theorem \ref{thm: new 7.2 - derogatory}
because the diagonal entries of 
	$W_j$ 
are $0$ for blocks 
	$m_{j2}$ to $m_{jq_j}$ 
but are nonzero for block 
	$m_{j1}$.  
Therefore, 
	$\sd\vp(\wt X) \supsetneq \rsd\vp(\wt X)$. 
\end{proof}

The following  result follows as an immediate consequence of 
Theorem~\ref{thm:vp is subdifferentially regular} and the preceding lemma.
\begin{theorem}\label{main result}
Let $f:\C\to \eR$ be proper, convex, lsc, and 
let $\vp$ be the associated spectral max function as in \eqref{eqn: vp defn - original}. 
Suppose that
	$\wt{X}\in\Cnn$ 
$f$ satisfies \eqref{f-twice-diffable-and-psd} at all active eigenvalues of $\wt{X}$.
Then $\vp$ is subdifferentially regular at $\wt X$ if and only if the active eigenvalues of $\wt X$ are nonderogatory.
\end{theorem}

\section{Application to the spectral radius}
\label{sec: radius}
Let $r: \C \to \R$ be the complex modulus map
	\( r(\zeta) := \abs{\zeta}   \).
The polynomial radius  $\sfr: \Pcal{n} \to \eR$
and the spectral radius  $\rho: \Cnn \to \eR$ are 
given by 
	\(	\sfr(p) = \max\{ r(\lam) \,|\, p(\lam) = 0 \}
	\)
and
	\(	
		\rho(X) = \max\{ r(\lam) \,|\, \det(\lam I - X) = 0 \},
	\)	
respectively. 
The results of the previous sections do not directly apply to $\rho$ as
\eqref{f-twice-diffable-and-psd} is not satisfied, since
	 \( r''(\lam; t\lam,t\lam)    =   0 \)    
for all $\lam\neq 0$ and $t\in \R$, and $r'(0)$ does not exist (see \S\ref{Derivatives}).
As in \cite{burke-lewis-overton:05a}, we overcome these hurdles by introducing the function
	\( r_2(\zeta) := \abs{\zeta}^2/2 . \)
Since
	$\nabla^2 r_2 (\lam)$ is the identity map on $\C$,
$r_2$  satisfies 
\eqref{f-twice-diffable-and-psd} on all of $\C$. 
Set 
	\( \rho_2(\wt X) := \max  \{   \abs{\lam}^2/2 \ \vert \ \det(\lam I - \wt X) = 0   \} . \)
The  strategy is to first consider the case 
$\rho(\wt X)>0$ using $\rho_2$ and then to directly consider the case 
 $\rho(\wt X)=0$.
When $\rho(\wt X)>0$,
the relationship between the regular subdifferentials of $\rho$ and $\rho_2$ 
is a consequence of the following result. 
\begin{lemma}\label{lem: rho and rho2} 
Let 
	$\tX \in  \Cnn$ 
be such that 
	$\rho(\wt X)>0$.
Then
	\[\begin{aligned}
		T_{\epi{\rho}}(\wt X,\mu)  &=  \{  (V,  \eta/\mu  )  \vert
					(V,\eta) \in T_{\epi{\rho_2}} (\wt X, \mu^2/2 )	
					  \}, 
	\ \mbox{and}\\
		\widehat{N}_{\epi{\rho}}(\wt X,\mu) 
			  &=
			\{  (W, \mu\tau )  \vert
			(W,\tau) \in \widehat{N}_{\epi{\rho_2}} (\wt X, \mu^2/2 )	
			 \}.
	\end{aligned}\]
\end{lemma}
\begin{proof} 
Recall the definitions of the tangent and regular normal cones given in \eqref{eqn: tangent cone defn} and 
\eqref{eqn: regular normal cone defn}, respectively.
The elementary proof of the tangent cone equality is identical to that in \cite[Lemma 7]{burke-lewis-overton:05a} 
with polynomials replaced by matrices and the root max function replaced by the spectral max function.
The regular normal cone expression follows by taking polars \eqref{eqn: regular normal cone defn}.
\end{proof}


\begin{theorem}[$\rho(\wt X)>0$]
\label{thm: new 7.2 - radius}
Let $\tX$ have Jordan form \eqref{eqn: P and J, 1} 
with
	$\rho(\tX) > 0$. 
Then 
	$Y\in\rsd \rho(\tX)$
if and only if $Y$ satisfies
	\eqref{W defn}-\eqref{eqn: Y and W, 2},
	\eqref{eqn: theorem 6.1 vanlsf}, 
and, for $j\in \Ical_\rho(\tX)$, 
the  diagonal entries $\theta_{j1}$ of $W_j$ satisfy 
	\begin{equation}
	\label{radius rsd 1} 
	\theta_{j1}/  \tlam_j   \in \R,  \ \ 
	\theta_{j1} /  \tlam_j  \geq 0, \ \mbox{ and } \
	 \mbox{$\sum_{j\in\Ical_\rho(\wt X)}$} n_j  \theta_{j1} | \tlam_j|  /   \tlam_j   =1,
	 \end{equation}
and, if in addition $m_j \geq 2$ (see \eqref{mj defn}), the  subdiagonal entries $\theta_{j2}$ of $W_j$ satisfy 
	\begin{equation}
	\label{radius rsd 2}
	\Realpart \left( \overline{ \theta_{j2} }    \tlam_j^2  \right) 
	\geq -  \theta_{j1}|\tlam_j|^2/  \tlam_j.  
	\end{equation}
%
Moreover,
	$Y \in \rsd\rho(\tX)^\infty$ if and only if
	$Y$ satisfies
	\eqref{W defn}-\eqref{eqn: Y and W, 2},
	$\theta_{j1}\!=\!0$ for all $j\in\{1, \dots m\}$, and, for 
	$j\!\in\! \Ical_\rho(\tX)$ with $m_j\!\geq\! 2$,
	$\Realpart( \overline{ \theta_{j2} }    \tlam_j^2 )\! \geq\! 0$.
Finally, $\rho$ is subdifferentially regular at $\tX$ if all active eigenvalues are nonderogatory.
\end{theorem}

\begin{remark}
\label{radius subderivative formula} 
The proof uses the following formula for the  subderivative of the polynomial radius when $p \in \Mcal{n}$ and $\sfr(p)>0$:
	\[
	\dee \sfr(\tilp)(v) 
	=
	\mbox{$\max_{j\in\Ical_\sfr(\tilp) }$}
	  \left \{ 
	 \left(\, \abs{\om_{j2}} - \Realpart ( \overline{ \tlam_j} \om_{j1}   ) \right)  /( \vert \tlam_j \vert n_j )
	  \right\},
	 \] 
where 
	$v=\sum_{j=1}^m \prod_{k\neq j} 
	\elem{n_k}{\tlam_k} 
	\sum_{s=1}^{n_j} 
	\om_{js}
	\elem{n_j-s}{\tlam_j}$,
provided that
	\( \om_{j2}    \in \cone( \tlam_j^2)   \)
	and
	\(  \om_{js}=0 \) for $s = 3, \dots, \ n_j $;
otherwise, 
	$d\sfr(\tilp)(v)=+\infty$.
\end{remark}

\begin{proof}
We first address the regular subdifferential formula.
We omit the full proof since it parallels the proof of Theorem~\ref{thm: new 7.2 - derogatory}.
We describe a few key ingredients specific to the radius.
First, the following formulas are needed (see \cite[Section 5]{burke-lewis-overton:05a})
	\begin{equation}
	\label{r gradient}
		\nabla r (\zeta)=\zeta/|\zeta| \quad \forall \ \zeta \in \C\setminus\{0\},
	\end{equation}
and
	\[
	 	r''(\zeta;\delta,\delta) = (1/|\delta|) \left(|\delta|^2 - \ip{ \zeta/|\zeta| }{\delta}_\C^2\right)
		\quad \forall \ \zeta, \delta \in \C\setminus\{0\}.
	\]
Applying these formulas to the quantity $\eta_j$ in \eqref{etaj defn} imply that
	\( \eta_j = 1 / | \tlam_j |, \)
	that display \eqref{eqn: sig 1} is equivalent to \eqref{radius rsd 1}, 
	and that display \eqref{eqn: theta f' inner product - max} is equivalent to \eqref{radius rsd 2}.
Second, note that the hypotheses of Theorems~\ref{theorems A B} and \ref{thm: section 6}, which are applied in the proof of
Theorem~\ref{thm: new 7.2 - derogatory},  do not require that $r$ satisfy \eqref{f-twice-diffable-and-psd}. 
They only require that $r$ be differentiable, which is the case on $\C\setminus\{0\}$. 
Finally, the explicit formula for the subderivative of $\sfr$ in Remark~\ref{radius subderivative formula} is utilized. 

Next, the formula for the recession cone of the regular subdifferential parallels that of 
Theorem~\ref{thm: recession cone}. 

Lastly, we discuss the subdifferential regularity in the case where $\tX$ has nonderogatory active eigenvalues and $\rho(\tX)>0$.
The map $r_2: \C \to \R$ satisfies \eqref{f-twice-diffable-and-psd} at $\lam$ for all $\lam \in \C$.
Therefore  $\rho_2$ is subdifferentially regular at $\tX$ by Theorem~\ref{thm:vp is subdifferentially regular}.
Using Lemma~\ref{lem: rho and rho2}, it is straightforward to establish the subdifferential regularity of $\rho$ (see \cite[Definitions 6.4 and 7.25]{rw:98a}---Clarke regularity of the epigraph is subdifferential regularity). 
\end{proof}

%

When
	$\rho(\tX) = 0$, $\wt X$ 
has a single eigenvalue, $0$, with algebraic multiplicity $n$. 
For this reason, we suppress the index ``$j$" in
	\eqref{W defn}-\eqref{eqn: Y and W, 2}.

\begin{theorem}[$\rho(\wt X)=0$]
\label{thm: rsd for rho at 0}
Let 
	$\wt X \in  \Cnn$ 
be such that
	$\rho(\wt X) = 0$.
Then 
	$Y \in \rsd \rho(\wt X)$ 
if and only if 
	$Y$ satisfies
	\eqref{W defn}-\eqref{eqn: Y and W, 2}
and
 	\begin{equation}
	\label{radius zero case diagonal}
	 	| \theta_1| \leq  1/n,
	\end{equation}
where $\theta_1$ is  the diagonal entry of $P^{-*}Y \wt P^*$.
Moreover, 
	$Y \in \rsd \rho(\wt X)^\infty$ if and only if 
$Y$ satisfies \eqref{W defn}-\eqref{eqn: Y and W, 2} and $\theta_1=0$. 
Finally, $\rho$ is subdifferentially regular at $\tX$  if $\tX$ is nonderogatory.  
\end{theorem}
\begin{proof} 
We first prove the regular subdifferential formula.
Note that the characteristic polynomial of $\tX$ is $\elem{n}{0}$.
We will use the fact that 
	\begin{equation}
	\label{eqn: iff for rsd}
	Y \in \rsd \rho(\wt X) \quad \mbox{ if and only if }\quad  \ip{Y}{Z} \leq \dee \rho(\wt X) (Z) \mbox{\ \  for all \  \ } Z\in  \Cnn .
	\end{equation}
Suppose 
	$Y \in  \Cnn $ 
satisfies 	\eqref{W defn}-\eqref{eqn: Y and W, 2}
and \eqref{radius zero case diagonal}.
By the chain rule \cite[Theorem 10.6]{rw:98a},
	\begin{equation}
	\label{eqn: chain rule - radius - 0}
		\dee \rho(\tX) (Z) 
		\geq \dee \sfr(\Phi_n( \tX))(	\Phi_n'(\tX)Z ) = \dee \sfr(\elem{n}{0})( 	\Phi_n'(\tX)Z  ) 
	\end{equation}
for all $Z \in \Cnn$.
By \cite[Theorem 11]{burke-lewis-overton:05a},
	\begin{equation}
	\label{eqn: subderivative sfr 0 - 2}
		\dee \sfr (\elem{n}{0}) ( \Phi_n'(\wt X)Z ) = \begin{cases}  \abs{\tr \left(V \right)}  /n
										& \mbox{ if }\tr  (N^{ s - 1} V  ) = 0 \mbox{ for } s=2,\dots, n \\
										+\infty & \mbox{otherwise},
							\end{cases}
	\end{equation}
where $V$ is given in \eqref{eqn: V defn for q} and $N\in\Cnn$ is the nilpotent matrix in \eqref{N sub bracket} (with $j$'s suppressed).
The inequality in   \eqref{eqn: iff for rsd} is trivially satisfied if 
	$\tr(N^{s-1}V) \neq 0$ 
for some 
	$s=2, \dots, n$ 
since in this case 
	$d\sfr(\elem{n}{0})(\Phi_n'(\tX)Z)$ 
is infinite, which implies  $d\rho(\tX)(Z)$ is infinite by \eqref{eqn: chain rule - radius - 0}.
So suppose  
	$\tr(N^{s-1} V) = 0$ 
for  
	$s=2, \dots, n$.
By
	\eqref{eqn: Y and W, 2}  
and
\eqref{eqn: theorem 6.1 vanlsf},
	\[
	\ip{Y}{Z}_{\Cnn} = \ip{W}{V}_{\Cnn} =  \sum_{\ell=1}^{n}  
	\Realpart
		 (
			\overline{ \theta_{  \ell }} \tr  ( N ^{\ell-1} V  )
		 )  , 
	\]
which simplifies to
	\[
	\Realpart
		 (
			\overline{ \theta_{  1 }} \tr  ( V  )
		 )   
		 = 		 
			\ip{  \theta_{  1 }}{ \tr  ( V  )}_\C.
	\]
Therefore,
	\begin{align*}
	\ip{Y}{Z}_{\Cnn}  
	&= \ip{\theta_1}{\tr(V)}_\C  \\
	& \leq \max_{\zeta \in \B}  \ip{\zeta/n}{\tr(V)}_\C   &&\mbox{by \eqref{radius zero case diagonal}}\\
	&=  \abs{\tr(V)}/n  \\
	&  = \dee \sfr(\elem{n}{0})(q(\lam)) &&  \mbox{by \eqref{eqn: subderivative sfr 0 - 2}} \\
	&   \leq \dee \rho(\wt X)(Z) &&\mbox{by \eqref{eqn: chain rule - radius - 0}}.
	\end{align*}
Condition \eqref{eqn: iff for rsd}  applied to the inequality above implies
 	$Y \in \rsd \rho(X)$.
 
Now suppose 
	$Y\in \rsd \rho(\wt X)$. 
By Theorem~\ref{theorems A B}, $Y$ satisfies \eqref{W defn}-\eqref{eqn: Y and W, 2}.
Choose 
	$Z:=\wt{P}^{-1}V\wt{P}$
with
	$V := \theta_1 I$.
By the definition of subderivative in \eqref{eqn: subderivative defn} (using $\rho(\wt X)=0$),
	\begin{align*}
	\dee \rho(\wt X)(Z) 
	&  = \liminf_{t \downarrow 0, \widetilde{Z}\to Z}  \rho(\wt X+t\widetilde{Z}  ) / t    \\
	&  \leq  \liminf_{t \downarrow 0}   \rho(\wt X+t Z ) /t   \\
	& = \liminf_{t \downarrow 0}   \rho(N+t V ) / t,   && \mbox{where $N$ is given in \eqref{N sub bracket} }\\
	&  = \liminf_{t \downarrow 0}   t\abs{ \theta_1}  / t,   && \mbox{since the only eigenvalue of $N+tV$ is $t\theta_1$, }  \\
	& =  \abs{\theta_1}. 
	\end{align*}
By combining \eqref{eqn: iff for rsd} and the inequality above,
	\[	\abs{\theta_1} 		  \geq \dee p(\wt X)(Z) 
						    \geq  \ip{Y}{Z}_{\Cnn}
						   = \langle {\wt P^{-*} Y \wt P^* },{V}   \rangle_{\Cnn}
						   = \Realpart( \overline{\theta_1} \tr(V) )  
						   = n \abs{\theta_1}^2,
	\]
which implies $1/n\geq\abs{\theta_1}$, giving \eqref{radius zero case diagonal}.

The formula for the recession cone of the regular subdifferential parallels that of 
Theorem~\ref{thm: recession cone}. 

Finally, we show that nonderogatory active eigenvalues are sufficient for regularity.
Suppose 
	$Y\in\sd\rho(\tX)$. 
Then, by Theorem \ref{theorems A B},
	\begin{equation}
	\label{radius Y}
	{\wt P^{-*}}Y{\wt P^{*}}=\sum_{\ell=1}^n\theta_\ell(N^*)^{\ell -1},
	\end{equation}
that is, $Y$ is lower triangular Toeplitz and has only one eigenvalue, $\theta_1$.
By definition of subgradient \eqref{general-sd-formula},
there exists 
	$\ds X^\nu \toarrow^\nu \wt X$ 
and 
	$Y^\nu \in \rsd\rho(X^\nu)$ 
with 
	$\ds Y^\nu\toarrow^\nu Y$.  
Since the set of nonderogatory matrices is open in $\Cnn$, we may without loss in generality assume that $X^\nu$ is nonderogatory
 for all $\nu$.	
If there is an infinite subsequence with 
	$\rho(X^\nu)=0$, 
then 
	$X^\nu = \wt X$ 
along this subsequence and there is nothing more to show.
So suppose 
	$\rho(X^\nu)>0$ 
for all $\nu$. Since there are only finitely many partitions of $n$, 
we may assume (by passing to a subsequence if necessary) that the distinct eigenvalues of 
	$X^\nu$ are  
	$\lam_1^\nu, \dots, \lam_m^\nu$ 
with multiplicities 
	$n_1, \dots, n_m$ 
with 
	$\lam_j^\nu \to 0$ 
for 
	$j=1, \dots, m$
and that the set of active indices 
	$\wt\Ical:=\Ical_\rho(X^\nu)$ 
is constant with respect to $\nu$. 
By Theorem~\ref{thm: new 7.2 - radius}, 
there exist nonsingular $P^\nu$ for all $\nu$ such that the diagonal of
	$(P^\nu)^{-*} Y^\nu (P^\nu)^* $ 
equals 
	$(\theta_{11}^\nu, \dots, \theta_{1n_1}^\nu, \dots, 
	\theta_{m1}^\nu, \dots, \theta_{mn_m}^\nu)$,   
which satisfy
	\begin{gather*}
		\theta_{j1}^\nu/  \lam_j^\nu  \!\in\! \R, 
		\    \theta_{j1} /  \lam_j^\nu  \geq 0, 
		 \mbox{ for }j \!\in\! \wt\Ical, 
		\  \
		\sum_{j\in\wt\Ical} n_j   \theta_{j1}^\nu \abs{ \lam_j^\nu}  /   \lam_j^\nu   \!=\!1, 
		\ \mbox{ and }  
		\theta_{j1}^\nu \!=\! 0
		\mbox{ for } j \!\notin\! \wt\Ical.
	\end{gather*}
Recall from \eqref{r gradient} that $\nabla r (\lam) = \lam/|\lam|$. 
By the compactness of $\B$, we may assume that there exist
	$\phi_j \in \B$ 
such that 
	$\nabla r (\lam_j^\nu) \to \phi_j$ for $j\in \wt \Ical$. 
The remainder of this proof follows that  of \cite[Theorem 7.3]{burke-eaton:12}.
Rewrite the summation in the previous displayed equation as
	\[
		\mbox{$\sum_{j\in\wt\Ical}$} \  n_j    \theta_{j1}^\nu/ \nabla r(\lam_j^\nu) = 1
    	\]
using \eqref{r gradient}. Define 
	\[
	 \gam_j^\nu := \begin{cases}  n_j    \theta_{j1}^\nu / \nabla r (\lam_j^\nu) , &\mbox{if } j\in \wt \Ical, \\
									0, &\mbox{if } j\notin\wt\Ical.\end{cases}
	\]
Then 
	$(\gam_1^\nu, \dots, \gam_m^\nu) \in \Delta^{m-1}$, the simplex defined in \eqref{simplex}.  
Since $\Delta^{m-1}$ is compact,
and by passing to a subsequence if necessary, there exists 
	$(\gam_1, \dots, \gam_m)\in \Delta^{m-1}$ such 
that
	$(\gam_1^\nu, \dots, \gam_m^\nu) \to (\gam_1, \dots, \gam_m)$.
The eigenvalues of $Y^\nu$ are 
	$\theta_{11}^\nu, \dots, \theta_{m1}^\nu$ 
with multiplicities 
	$n_1, \dots, n_m$, 
and they converge
to $\theta_1$, the only eigenvalue of $Y$.
Therefore, 
	\[
		\lim_{\nu \to \infty} \gam_j^\nu \nabla r (\lam_j^\nu)  = \gam_j \phi_j =   n_j \theta_1
	\]
for $j=1, \dots, m$.
Summing over $j$, we have
	\[
		\mbox{$\sum_{j=1}^m$} \gam_j \phi_j = 	\mbox{$\sum_{j=1}^m$}  n_j \theta_1 = \theta_1 \mbox{$\sum_{j=1}^m$} n_j = \theta_1 n .
	\]
Since 
	$(\gam_1,\dots,\gam_m)\in\Delta^{m-1}$, 
and 
	$\phi_j \in \B$ for all $j$,  
	\[ 
		|\theta_1| n 
		=  \left| \mbox{$\sum_{j=1}^m$} \gam_j \phi_j  \right| 
		\leq     \mbox{$\sum_{j=1}^m$} \gam_j |\phi_j| 
		\leq \mbox{$\sum_{j=1}^m$} \gam_j = 1,
	\] 
which implies 
	$|\theta_1|  \leq 1/n$.

Now suppose $Y\in\sd^\infty\rho(\tX)$.
By Theorem~\ref{theorems A B}, representation \eqref{radius Y} holds.
By definition of horizon subgradient \eqref{horizon-subdifferential},
there exist 
	$\ds X^\nu \toarrow^\nu \wt X$, 
	$s^\nu \downarrow 0$ 
and 
	$Y^\nu \in \rsd\rho(X^\nu)$ 
with 
	$\ds s^\nu Y^\nu\toarrow^\nu Y$.  
Set 
	$V :=  \theta_1 I$ 
and  
	$Z := \wt P^{-1} V \wt P$ . 
Computing the inner product,
	\begin{equation}
	\label{ip zero}
		\langle Y, (\wt X + tZ)-\wt X \rangle_{\Cnn} 	 
					   = t \ip{Y}{Z}_{\Cnn}  
					   = t\ip{W}{V}_{\Cnn}
					   =  t n \abs{\theta_1}^2  .
	\end{equation}
Since $Y^\nu \in \rsd\rho(X^\nu)$, 
	\[
		\rho(X^\nu  + tZ)  -  \rho(X^\nu) 
	 	 \geq 
			\ip{Y^\nu }{(X^\nu + tZ) - X^\nu}_{\Cnn}  +  o(||tZ||),
	\]
for all $t>0$, which implies
	\[
		s^\nu(\rho(X^\nu  + tZ)  -  \rho(X^\nu)) 
			  \geq  \ip{s^\nu Y^\nu\!}{(X^\nu  +  tZ)  -  X^\nu}_{\Cnn}  +  s^\nu o(||tZ||)
	\]		
for all $t>0$. Taking the limit as $\nu\to\infty$ gives
	\begin{align*}
		0    	& \geq  	\langle{Y},{(\wt X  +  tZ)  -  \tX}\rangle_{\Cnn}, &&\mbox{which, by \eqref{ip zero}, implies} \\
		0 	& \geq	t n \abs{\theta_1}^2     \geq 0, 	
	\end{align*}
that is, $\theta_1=0$.
\end{proof}

\begin{theorem}[Derogatory active eigenvalues]
Suppose 
	$\wt X \in\Cnn$ 
has a derogatory active eigenvalue.
Then 
	$\sd\rho(\tX) \supsetneq \rsd\rho(\tX)$.
\end{theorem}
\begin{proof} Let $\tX\in\Cnn$. For the case $\rho (\tX)>0$, see the proof of Theorem \ref{thm: new 8.2}. 
Suppose $\rho(\tX)=0$ and that the eigenvalue $0$ is derogatory.
Let $\eps^\nu \downarrow 0$ and set
	\( X^\nu := \wt P(J + \eps^\nu E)\wt P^{-1}, \)
where $J$ and $\wt P$ are defined in \eqref{eqn: P and J, 1}  
(with the index $j$ suppressed and where the matrix $\wt B$ does not appear) 
and $E$ is one in the first $m_1$ diagonal positions corresponding to block $J_1$ and is zero elsewhere. 
The matrix $X^\nu$ has exactly one active eigenvalue, $\eps^\nu$, with multiplicity $m_1$, and 
	\( \wt P(J + \eps^\nu E) \wt P^{-1} \) 
is the Jordan form of $X^\nu$.
 By Theorem \ref{thm: new 7.2 - radius}, $\rsd\rho(X^\nu)$ includes the matrix
	\( (1/m_1) \wt P^{-*} E \wt P^* \)
for all $\nu$. 
Therefore, 
	\(  (1/m_1) \wt P^{-*} E \wt P^* \in \sd\rho(\tX) . \) 
Yet by Theorem~\ref{theorems A B},
	\( (1/m_1) \wt P^{-*} E \wt P^* \notin \rsd\rho(\tX) \)
since its diagonal entries are zero for blocks $m_2$ to $m_q$ but nonzero for block $m_1$, i.e. the diagonal entries are not equal.
Therefore, 
	$\sd\rho(\tX) \supsetneq \rsd\rho(\tX)$ 
if $\tX$ has a derogatory active eigenvalue. 
\end{proof}

\begin{theorem}
The spectral radius is subdifferentially regular at $\tX\in\Cnn$ if and only if all active eigenvalues of $\tX$ are nonderogatory.
\end{theorem}

\subsection{Examples}
\label{examples}

First consider the matrix
\[
		A =  \left[\begin{array}{ccc}
				1 & 1 & 0 \\
				0 & 1 & 0 \\
				0  & 0 & -1
				\end{array}\right], 
	\] 
which has two nonderogatory active eigenvalues, $1$ and $-1$. 
By Theorem~\ref{thm: new 7.2 - radius},
$Y \in \sd \rho(A)$ if and only if there exist $\theta_{11}, \theta_{12}, \theta_{21} \in \C$ such that
	\[
	Y =  
 		\left[ 
		\begin{array}{c c c}
		 \theta_{11} &  0 			& 0  \\ 
		\theta_{12} & \theta_{11}	& 0 \\
				0  &     		0	& \theta_{21} \\
		\end{array} \right] ,
	\]
where
	\( \Realpart(   \theta_{12}    ) \geq -  \theta_{11}\), \ \
	\(  \theta_{11}  \in [0,\infty)\), \ \
	 \(\theta_{21}   \in (-\infty,0]\),  \ \
and 
	  \(	2\theta_{11}  - \theta_{21} = 1\).
Second consider the matrix\footnote{The matrix $B-I$ is in the family of matrices considered in \cite{grundel:13} for the 
spectral abscissa.} 
	\[		B=  \left[\begin{array}{ccc}
				1 & 1 & 0 \\
				0 & 1 & 0 \\
				0  & 0 & 1
				\end{array}\right].
	\] 
By Theorem~\ref{thm: new 7.2 - radius},
$Y \in \rsd \rho(B)$ if and only if  
	\begin{equation}
	\label{eqn: der case ex}
	Y =  \frac{1}{3}I +
		\left[ 
		\begin{array}{c c c}
		0 & 0 & 0  \\ 
		\theta & 0& 0 \\
		0  &0	& 0 
		\end{array} \right],
	\end{equation}
where
	\( \Realpart(   \theta     ) \geq - 1/3 \). 
Next consider the sequence of matrices 
	\[ B^\nu := B + \Diag(0,0,1/\nu) \]
for 
	$\nu \in \mathbb{N}$.
The only active eigenvalue of $B^\nu$ is $(1+1/\nu)$, which is nonderogatory.
By Theorem~\ref{thm: new 7.2 - radius}, 
	$ M:=\Diag(0,0,1) \in \rsd \rho (B^\nu)$ 
for all $\nu$, 
which implies 
	$M\in \sd\rho(B)$, 
yet 	
	$M \notin\rsd\rho(B)$ 
since it does not satisfy the form given in \eqref{eqn: der case ex}, i.e.  the diagonal entries are not equal.

\section{Summary}
This paper extends the variational results for the spectral abscissa mapping in \cite{burke-overton:01b}
to convexly generated spectral max functions. 
Two very different methods of analysis are applied.
The first uses the Arnold form (\S\ref{sec: arnold}) and tools from \cite{burke-lewis-overton:01a}. 
A nonsmooth chain rule is applied to the composition of the characteristic polynomial mapping
and a max root function for polynomials \eqref{diagram}.
The subdifferential theory for polynomial max root functions is developed in \cite{burke-eaton:12},
which in turn builds on the work in \cite{burke-lewis-overton:04a,burke-lewis-overton:05a}. 
The key technical breakthroughs for the first approach appear in \S\ref{subsec:char poly} culminating in Theorem \ref{thm: G is injective},
which describes the variational behavior of the mapping $G$ taking a matrix to its distinct active monomial factors on the space $\Scal_{\tilp}$.
This yields our first main result, Theorem~\ref{thm:vp is subdifferentially regular},
which gives a formula for the subdifferential and establishes subdifferential regularity when all active eigenvalues are nonderogatory.

The major drawback to Theorem~\ref{thm:vp is subdifferentially regular} 
is the requirement that all of the active eigenvalues be nonderogatory. 
It is this hypothesis that gives access to our polynomial results through Theorem \ref{thm: G is injective}. 
Our second line of attack avoids the polynomial results by appealing directly to underlying matrix structures. 
This approach extends results in \cite{burke-overton:01b} to convexly generated spectral max functions when possible. In Theorem \ref{thm: new 7.2 - derogatory}, we characterize the regular subgradients of a convexly generated spectral max
function without assuming nonderogatory active eigenvalues.
In addition, Lemma \ref{thm: new 8.2}  shows that nonderogatory
active eigenvalues are necessary for subdifferential regularity. Combined with our earlier results, 
we obtain Theorem \ref{main result}, which shows that subdifferential regularity 
occurs if and only if all active eigenvalues are nonderogatory. This neatly extends the 2001 result of
Burke and Overton for the spectral abscissa \cite{burke-overton:01b} to the class
of convexly generated spectral max functions satisfying \eqref{f-twice-diffable-and-psd}.

The results are applied in Section \ref{sec: radius} to 
obtain new variational properties for the spectral radius mapping.

\bibliographystyle{plain}

\end{document}